\documentclass[reqno]{amsart}
\usepackage{amsthm,amsmath,amssymb}
\usepackage{stmaryrd,mathrsfs}
\usepackage[T1]{fontenc}
\usepackage{esint}
\usepackage{tikz}

\allowdisplaybreaks

\newcommand{\ud}[0]{\,\mathrm{d}}

\newcommand{\dist}[0]{\operatorname{dist}}

\newcommand{\abs}[1]{|#1|}

\newcommand{\Norm}[2]{\|#1\|_{#2}}

\newcommand{\pair}[2]{\langle #1,#2 \rangle}

\newcommand{\ave}[1]{\langle #1\rangle}

\newcommand{\lspan}[0]{\operatorname{span}}

\newcommand{\bddlin}[0]{\mathscr{L}}

\newcommand{\BMO}[0]{\operatorname{BMO}}

\newcommand{\R}{\mathbb{R}}
\newcommand{\C}{\mathbb{C}}
\newcommand{\N}{\mathbb{N}}
\newcommand{\Z}{\mathbb{Z}}

\renewcommand{\P}[0]{\mathbb{P}}
\newcommand{\E}[0]{\mathbb{E}}

\swapnumbers \numberwithin{equation}{section}

\theoremstyle{plain}
\newtheorem{theorem}[equation]{Theorem}

\newtheorem{corollary}[equation]{Corollary}
\newtheorem{lemma}[equation]{Lemma}

\theoremstyle{definition}
\newtheorem{definition}[equation]{Definition}

\theoremstyle{remark}
\newtheorem{remark}[equation]{Remark}

\makeatletter
\@namedef{subjclassname@2010}{%
  \textup{2010} Mathematics Subject Classification}
\makeatother

\begin{document}

\title{On the sense of convergence in the dyadic representation theorem}

\author{Tuomas Hyt\"onen}
\address{Department of Mathematics and Statistics, P.O. Box 68 (Pietari Kalmin katu 5), FI-00014 University of Helsinki, Finland}
\email{tuomas.hytonen@helsinki.fi}
\address{(Address as of 1 Jan 2024:) Department of Mathematics and Systems Analysis, Aalto University, P.O. Box 11100, FI-00076 Aalto, Finland}
\email{tuomas.p.hytonen@aalto.fi}


\thanks{The author was supported by the Academy of Finland through project No. 346314 (Finnish Centre of Excellence in Randomness and Structures ``FiRST'')}


\maketitle


\begin{abstract}
The dyadic representation of any singular integral operator, as an average of dyadic model operators, has found many applications. While for many purposes it is enough to have such a representation for a ``suitable class'' of test functions, we show that, under quite general assumptions (essentially minimal ones to make sense of the formula), the representation is actually valid for all pairs $(f,g)\in L^p(\R^d)\times L^{p'}(\R^d)$, not just test functions.
\end{abstract}


\section{Introduction}

The dyadic representation of Calder\'on--Zygmund operators has its roots in the works of Figiel \cite{Figiel:90}, Nazarov, Treil and Volberg \cite{NTV:Tb}, and Petermichl \cite{Pet}. 
A prototype of several recent versions was established by the author as a tool to settle the $A_2$ conjecture about the sharp weighted bounds for singular integrals \cite{Hytonen:A2}; other applications of dyadic representation include estimates for Banach space -valued extensions (e.g., \cite{PS:2014}) and commutators (e.g., \cite{HLW:17}) of singular integrals.

For the typical applications, it suffices to have the representation formula for a suitable dense class of test functions. The representation then leads to favourable estimates for the same functions, and these estimates can then be extended to general functions by density. However, the dyadic representation may also be seen as an independent structure theorem for singular integral operators, and from this point of view it is natural to inquire about the maximal generality and the sense in which such a formula is valid. Besides intrinsic interest, this should streamline applications of the formula to new situations, hopefully reducing the need of less interesting and potentially tedious approximation and density considerations as intermediate steps.

The goal of this paper is to establish a version of the general dyadic representation theorem of \cite{Hytonen:A2} under essentially minimal assumptions on both the operator that we wish to decompose and the functions on which it acts. A related study of the sense of convergence of Petermichl's special dyadic representation for the Hilbert transform \cite{Pet} was carried out in \cite{Hytonen:Pet}. Somewhat more distant cousins of the present work are the investigations by Wilson \cite{Wilson:converge,Wilson:convergeH1} of the sense of convergence of the classical Calder\'on reproducing formula.

Before going into the present contribution, let us review the existing dyadic representations in some more detail. In the original application of \cite{Hytonen:A2}, the interest was in quantifying the norms of operators that are already known to be bounded qualitatively, and hence the representation theorem was stated and proved under the assumption of $L^2(\R^d)$ (or weighted $L^2(w)$) boundedness of the operators in question. However, it was subsequently observed that, with little modification, the same ideas can also be used to prove this $L^2(\R^d)$-boundedness from \emph{a priori} weaker assumptions in the style of the $T(1)$ theorem of David and Journ\'e \cite{DJ:T1} and thus, if fact, to give another proof of their result  \cite{Hytonen:Expo}.

This new approach to the $T(1)$ theorem via a dyadic representation formula has turned out to be quite flexible in extending the $T(1)$  theorem to other situations. Thanks to Martikainen's dyadic representation on product spaces \cite{Mart:12}, Journ\'e's $T(1)$ theorem \cite{Journe:85} for bi-parameter singular integrals was obtained with simplified assumptions by Martikainen \cite{Mart:12} and recently extended to matrix-weighted spaces by Domelevo, Kakaroumpas, Petermichl, and Soler~i Gibert \cite{DKPS:24}. As another recent example, Li, Martikainen, Vuorinen, and the present author \cite{HLMV:Zygmund} found a dyadic representation for multiparameter singular integrals adapted to so-called Zygmund dilations and used this to obtain the first $T(1)$ theorem in this setting, where previous results, notably by \cite{FP:97,HLLT:19}, were restricted to convolution-type operators.

In writing down representation formulas for singular operators that are only weakly defined to begin with, as in the applications to $T(1)$ theorems just mentioned, some issues need to be addressed:
\begin{enumerate}
  \item\label{it:init} What is the initial class of test functions on which our operator acts?
  \item\label{it:reprInit} Is the prospective representation formula valid for all functions in this class?
  \item\label{it:bdExt} Does the representation formula provide estimates that allow a bounded extension of the operator from the initial class to a larger space of interest?
  \item\label{it:reprExt} Once the operator has been extended, does the representation formula remain valid for all functions in the larger space?
\end{enumerate}
As far as questions \eqref{it:reprInit} and \eqref{it:reprExt} are concerned, the complexity of the representation formula is also of some interest. As a rule of thumb, the simpler the representation, the more flexible it is to apply to new instances of \eqref{it:bdExt}.
Two key elements of all results that we call ``dyadic representation theorems'' are: (a) an expectation over a random choice of the underlying dyadic system and (b), for each such dyadic system, a series of ``dyadic singular integrals'' adapted to this particular system. However, an additional element that may or may not be present in different versions is: (c) an initial truncation of the dyadic operators (say, to boundedly many scales of dyadic cubes only), combined with a limit over the truncation parameter in the end. Thus the representation formula takes the form of either
\begin{equation}\label{eq:DRTsimple}
  \pair{Tf}{g}=\E\sum_k \pair{S_{k,\mathscr D} f}{g}
\end{equation}
or
\begin{equation}\label{eq:DRTlimit}
  \pair{Tf}{g}=\lim_{n\to\infty}\E^{(n)}\sum_k \pair{S_{k,\mathscr D}^{(n)} f}{g}
\end{equation}
for appropriate dyadic operators $S_{k,\mathscr D}$ (which also come in a variety of forms, but we refrain from dwelling too deep into this) or their truncated forms $S_{k,\mathscr D}^{(n)}$. In \eqref{eq:DRTlimit}, we also allow for the possibility that the probability space over which the average is taken may depend on the truncation parameter $n$. If $\E^{(n)}\equiv\E$ is independent of $n$, then \eqref{eq:DRTsimple} could be deduced from \eqref{eq:DRTlimit} provided that each $S_{k,\mathscr D}$ is the weak limit of its truncated versions and if, moreover, we can justify the exchange of the limit with the expectation and the sum (e.g. by dominated convergence).

The prototype dyadic representation of \cite[Theorem 3.1]{Hytonen:A2} was already stated for all $f\in L^2(\sigma)$ and $g\in L^2(w)$ (where $w$ is any Muckenhoupt $A_2$-weight and $\sigma=1/w$ the corresponding dual weight), thus in particular for $f,g\in L^2(\R^d)$, but it takes the semi-complicated form \eqref{eq:DRTlimit} with $\E^{(n)}=\E$. In \cite{HPTV}, a simpler formula of type \eqref{eq:DRTsimple} was achieved; the formulation of \cite[Theorem 4.1]{HPTV} as an operator identity $T=\E\sum_k S_{k,\mathscr D}$  is somewhat bold, but in the proof \cite[Sec. 4.3]{HPTV} it is explained that the identities should be ``understood in the weak sense, as equalities of the bilinear forms for $f, g \in C_0^\infty$''. A similar approach, with $f\in C_0^1(\R^d)\otimes X$ and $g\in C_0^1(\R^d)\otimes Y^*$ for Banach spaces $X$ and $Y$, is followed in \cite{HH:2016}, where the representation theorem is extended to Banach space \mbox{-valued} functions. In all these results, the $L^2(\R^d)$ boundedness or similar is assumed as a prerequisite for the representation theorem.

As for the $T(1)$-type theorems, where a representation formula is used to deduce the $L^2(\R^d)$-boundedness (or similar) from {\em a priori} weaker assumptions, a representation of type \eqref{eq:DRTlimit} is often employed, e.g. by \cite{GH:2018,Volberg:15}. Here one uses dyadic test functions that are constant on sufficiently small dyadic cubes (say, those of side-length $2^{-n}$), and randomises only over the larger scales of cubes (cf. \cite[\S 2.1]{GH:2018} or \cite[\S 2]{Volberg:15}). Thus, on the one hand, the limit in \eqref{eq:DRTlimit} trivialises in some sense (for each pair of dyadic test functions $f$ and $g$, the quantity inside the limit is constant for all $n\geq n(f,g)$). But, on the other hand, the expectation $\E^{(n)}$ is genuinely $n$-dependent, which is an obstacle to reducing \eqref{eq:DRTlimit} to the simpler form \eqref{eq:DRTsimple}. This approach is also followed in the recent monograph \cite[Chapter 12]{HNVW3}.

The goal of this paper is to obtain a representation formula of the simpler type \eqref{eq:DRTsimple} and to provide a fairly satisfactory answer, in our opinion, to questions \eqref{it:init} through \eqref{it:reprExt} above. Roughly speaking, our answer reads as follows:
\begin{enumerate}
  \item\label{ans:init} The initial test function space $\mathscr I$ consists of all linear combinations of indicators of axes-parallel rectangles in $\R^d$.
  \item\label{ans:reprInit} We obtain a representation formula of type \eqref{eq:DRTlimit} for all $f,g\in\mathscr I$, provided that $T$ is a weakly defined singular integral operator with standard Calder\'on--Zygmund kernel bounds, the usual assumptions of the (dyadic) $T(1)$ theorem, plus a slight strengthening of the usual (dyadic) weak boundedness property and a technical measurability condition.
  \item\label{ans:bdExt} The representation gives the extension of $T$ to a bounded operator on $L^p(\R^d)$ for $p\in(1,\infty)$.
  \item\label{ans:reprExt} The representation remains valid for all $f\in L^p(\R^d)$ and $g\in L^{p'}(\R^d)$.
\end{enumerate}
We refer the reader to Theorem \ref{thm:T1} for the precise formulation. 

The rest of this paper is structured as follows. In the short \S\ref{sec:test}, we provide some heuristics to justify our choice of the test functions in answer \eqref{ans:init}. In \S\ref{sec:setup}, we recall the relevant definitions of weakly defined singular integral operators and Calder\'on--Zygmund operators with Dini-continuous kernels, give a statement of our main result in Theorem \ref{thm:T1}, and make several comments about its assumptions in relation to previous literature. The proof of this theorem is then divided over the remaining sections of the paper. Our main emphasis consists of justifying the details of the identities involved in the representation, including the relevant measurability and convergence issues. The related estimates for various terms of the representation can be mostly quoted from the existing literature, where we mostly refer the reader to relevant places in \cite[Chapter 12]{HNVW3} and only repeat a few selected steps for the sake of readability. In \S\ref{sec:BCR}, we establish an initial decomposition by means of the so-called BCR (Beylkin--Coifman--Rochberg) algorithm. The main novelty here is that we make a finite version of the decomposition, with an explicit expression for the error term that is absent in the usual infinite version. The error term is then analysed in \S\ref{sec:error}. In \S\ref{sec:main}, we turn to the main term and isolate its ``diagonal'' and ``off-diagonal'' parts. Up to this point, the analysis takes place relative to a fixed dyadic system, but the next step requires an averaging over a random choice of the dyadic system. To make this rigorous, we dedicate \S\ref{sec:meas} to relevant measurability issues. These consideration are largely absent in the previous literature; in particular, they are avoided in the finitary model of \cite{GH:2018,Volberg:15} and \cite[Chapter]{HNVW3} but at the price of giving only the weaker representation of type \eqref{eq:DRTlimit}. With the measurability issues settled, we carry out the averaging process in \S\ref{sec:off}, identifying the average of the off-diagonal part with a series of (generalised) dyadic shifts. We note that the precise form of our shifts is that introduced by Grau de la Herr\'{a}n and the author in \cite{GH:2018}, which somewhat deviates from the original definition of \cite{Hytonen:A2}. These new shifts now seem to provide the most efficient representation currently available, manifested in at least two ways: the representation involves a single infinite series over a ``complexity'' parameter $k\in\N$, in contrast to a double series over $(m,n)\in\N^2$ in \cite{Hytonen:A2}, and this representation can be shown to converge under minimal Dini-continuity assumptions of the kernel, in contrast to the standard H\"older continuity assumptions used in \cite{Hytonen:A2}. This point has been recently made and exploited by Airta et al. \cite{AMV:prod,AMV:UMD}. In \S\ref{sec:concl}, we put the pieces together to complete the proof of the main Theorem \ref{thm:T1}. In the final \S\ref{sec:UMD}, we make some comments about the extension of these results to the setting of Banach space -valued functions.


\section{Choice of the test functions}\label{sec:test}

In the most general terms, we would like to have a framework, where a bilinear form $\tau(f,g)$ is initially defined for some pairs of test functions $f$ and $g$; we would then like to make estimates that allow us to conclude that the action of $\tau$ can be boundedly extended to pairs of functions in a larger class like $L^p(\R^d)\times L^{p'}(\R^d)$. At this level of generality, various choices are certainly possible, but to be able to make any use of the bilinear nature of $\tau$ with expansions of the functions, it seems reasonable to at least require that the class of test functions be a vector space.

To be a bit more specific, we would like to expand $f$ and $g$ in terms of a Haar basis related to a system of dyadic cubes $\mathscr D$. This leads to the natural restriction that our test function space should at least contain all (linear combinations of) indicators $1_Q$ of cubes $Q\in\mathscr D$.

Moreover, following the framework of \cite{NTV:Tb}, we would like to consider a random choice of the dyadic system $\mathscr D$. Now, every cube (by which we always understand an axes-parallel cube) $Q$ of sidelength $\ell(Q)\in 2^{\Z}$ will be a member of some dyadic system $\mathscr D$ that may arise via this random choice; thus our test function space should contain the indicators of all cubes $Q$ with $\ell(Q)\in 2^{\Z}$. But any (axes-parallel) rectangle $R$ can be realised as the intersection $P\cap Q$ of some cubes with $\ell(P)=\ell(Q)\in 2^{\Z}$. Although $1_{P\cap Q}=1_P\cdot 1_Q$ is not in the linear span of $1_P$ and $1_Q$, it is clearly closely related. For instance, the function $1_P-1_Q=1_{P\setminus Q}-1_{Q\setminus P}$ belongs to the span; if we impose the mild additional assumption that the test function class should also be stable under positive and negative parts, then $1_{P\setminus Q}$ and hence $1_P-1_{P\setminus Q}=1_{P\cap Q}=1_R$ will also be admissible test functions.

This discussion motivates the following:

\begin{definition}\label{def:testf}
The test function classes $\mathscr I$ and $\mathscr I_0$ are defined as
\begin{equation*}
  \mathscr I :=\lspan\{1_R:R\subseteq\R^d\text{ is a rectangle}\}, \qquad
  \mathscr I_0 :=\Big\{h\in\mathscr I:\int h=0\Big\}.
\end{equation*}
If $\mathscr D$ is a fixed system of dyadic cubes, we also denote
\begin{equation*}
  \mathscr I_{\mathscr D}:=\lspan\{1_Q:Q\in\mathscr D\}\subsetneq \mathscr I,\qquad\mathscr I_{\mathscr D,0}:=\mathscr I_{\mathscr D}\cap\mathscr I_0.
\end{equation*}
\end{definition}

\section{Set-up and statement of the representation theorem}\label{sec:setup}

For convenience when working with cubes, we denote by $\abs{\ }$ the $\ell^\infty$ (and not the Euclidean $\ell^2$) norm of $\R^d$.

\begin{definition}\label{def:CZK}
Let $\dot\R^{2d}:=\{(x,y)\in\R^{2d}:x\neq y\}$. Let $\omega:[0,\frac12]\to[0,\infty)$ be increasing.
We say that $K:\dot\R^{2d}\to\C$ is an $\omega$-Calder\'on--Zygmund kernel, if
\begin{equation*}
  \abs{K(x,y)}\leq\frac{c_K}{\abs{x-y}^d},
\end{equation*}
for all $(x,y)\in\dot\R^{2d}$, and
\begin{equation*}
  \abs{K(x,y)-K(x',y)}+\abs{K(y,x)-K(y,x')}\leq\omega\Big(\frac{\abs{x-x'}}{\abs{x-y}}\Big)\frac{1}{\abs{x-y}^d},
\end{equation*}
whenever $\abs{x-x'}<\frac12\abs{x-y}$.
\end{definition}

We denote
\begin{equation*}
  \Norm{\omega}{\operatorname{Dini}^s}:=\int_0^{\frac12}\omega(u)(\log\frac1u)^s\frac{\ud u}{u},\quad
  \Norm{\omega}{\operatorname{Dini}}:=\Norm{\omega}{\operatorname{Dini}^0}.
\end{equation*}

\begin{definition}\label{def:tau}
We say that $\tau:\mathscr I\times\mathscr I\to\C$ is a weakly defined $\omega$-Calder\'on--Zygmund operator with kernel $K$, if $K$ is an $\omega$-Calder\'on--Zygmund kernel and
\begin{equation}\label{eq:kernelRep}
  \tau(1_R,1_S)=\iint K(x,y)1_R(y)1_S(x)\ud y\ud x
\end{equation}
whenever $R,S\subseteq\R^d$ are disjoint rectangles.
\end{definition}

It is elementary to check that $\abs{x-y}^{-d}$ is integrable over a product of disjoint rectangles, and hence the integral in \eqref{eq:kernelRep} exists under the stated assumptions.

\begin{lemma}\label{lem:T1}
If $\tau$ is a weakly defined $\omega$-Calder\'on--Zygmund operator with $\Norm{\omega}{\operatorname{Dini}^0}<\infty$, and $h\in\mathscr I_0$, then the expression
\begin{equation}\label{eq:T1}
   \tau(1_{3Q},h)+\iint [K(x,y)-K(z_Q,y)]1_{(3Q)^c}(y)h(x)\ud x\ud y
\end{equation}
is well-defined for any cube $Q$ (with centre $z_Q$) that contains the support of $h$, and its value is independent of the choice of such a cube.
\end{lemma}

\begin{definition}\label{def:T1}
Under the assumptions of Lemma \ref{lem:T1}, we denote the value of \eqref{eq:T1} by $\tau(1,h)$. The expression $\tau(h,1)$ is defined analogously.
\end{definition}

\begin{theorem}\label{thm:T1}
Let $\tau:\mathscr I\times\mathscr I\to\C$ be a weakly defined $\omega$-Calder\'on--Zygmund operator.
Let $p\in(1,\infty)$, denote $p^*:=\max(p,p')$, and suppose that $\Norm{\omega}{\operatorname{Dini}^{1-\frac{1}{p^*}}}<\infty$. Then the following conditions are equivalent:
\begin{enumerate}
  \item\label{it:T1bd} There is a bounded linear operator $T\in\bddlin(L^p(\R^d))$ such that $\tau(f,g)=\pair{Tf}{g}$ for all $f,g\in\mathscr I$.
  \item\label{it:T1cont} The bilinear form $\tau$ satisfies:
  \begin{enumerate}
    \item\label{it:T1T1} the $T(1)$ conditions: for some $b_1,b_2\in\BMO(\R^d)$ and all $h\in\mathscr I_0$,
\begin{equation*}
  \tau(1,h)=\pair{b_1}{h},\qquad\tau(h,1)=\pair{h}{b_2};
\end{equation*}    
    \item\label{it:T1wbp} the weak boundedness property $\abs{\tau(1_Q,1_Q)}\leq\Norm{\tau}{wbp}\abs{Q}$ for all cubes $Q\subseteq\R^d$;
    \item\label{it:T1swbp} the semi-weak boundedness property $\abs{\tau(1_R,1_S)}\leq\Norm{\tau}{swbp}\abs{Q}$ whenever $R,S\subseteq Q$ are rectangles contained in a cube $Q\subseteq\R^d$;
    \item\label{it:T1contcont} the weak continuity property that $z\in\R^d\mapsto\tau(1_{Q+z},1_{Q+z})$ is continuous for all cubes $Q\subseteq\R^d$.
  \end{enumerate}
  \item\label{it:T1meas} The form $\tau$ satisfies conditions as in \eqref{it:T1cont}, except that \eqref{it:T1contcont} is replaced by
  \begin{enumerate}\setcounter{enumii}{4}
  \item\label{it:T1measmeas}   the weak measurability property that  $z\in\R^d\mapsto\tau(1_{Q+z},1_{Q+z})$ is measurable for all cubes $Q\subseteq\R^d$.
  \end{enumerate}
\end{enumerate}
Under these equivalent conditions, we have the representation formula
\begin{equation}\label{eq:DRTmain}
\begin{split}
  \pair{Tf}{g}
  &=\E\Big(\pair{\mathfrak H_{\mathscr D}f}{g}+\pair{\Pi_{b_1,\mathscr D}f}{g}+\pair{\Pi_{b_2,\mathscr D}^*f}{g}\Big) \\
  &\qquad+\E\sum_{\gamma\in\{0,1\}^2\setminus\{0\}}\sum_{k=2}^\infty\omega(2^{-k})\pair{A^{(\gamma,k)}_{\mathscr D}f}{g},
\end{split}
\end{equation}
valid for all $f\in L^p(\R^d)$ and $g\in L^{p'}(\R^d)$, where $\mathfrak H_{\mathscr D}$ is a Haar multiplier, $\Pi_{b_i,\mathscr D}$ ($i=1,2$) are dyadic paraproducts, and $A^{(\gamma,k)}_{\mathscr D}$ are certain generalised dyadic shifts of order $k$ and type $\gamma$, each related to a dyadic system $\mathscr D$, and $\E$ is the expectation over a random choice of the dyadic system. (See Section \ref{sec:concl} for their definitions.) For each fixed $\mathscr D$, the series over $k\geq 2$ in \eqref{eq:DRTmain} converges absolutely, and this convergence is uniform with respect to the choice of $\mathscr D$.
\end{theorem}

\begin{remark}\label{rem:p2}
For $p=2$, the assumption on $\omega$ becomes $\Norm{\omega}{\operatorname{Dini}^{\frac12}}<\infty$, which seems to be the weakest known regularity assumption to run any proof of the $T(1)$ theorem. Of course, as soon as we know that $T\in\bddlin(L^2(\R^d))$, we also obtain $T\in\bddlin(L^p(\R^d))$ for all $p\in(1,\infty)$ by classical Calder\'on--Zygmund theory, and the stronger assumption $\Norm{\omega}{\operatorname{Dini}^{1-\frac{1}{p^*}}}<\infty$ is not needed for this conclusion about the boundedness of the operator alone. The point of this stronger assumption is to guarantee the validity of the dyadic representation formula \eqref{eq:DRTmain} for all $(f,g)\in L^p(\R^d)\times L^{p'}(\R^d)$ with arbitrary $p\in(1,\infty)$. This might be useful to know, given that any condition of the type $\Norm{\omega}{\operatorname{Dini}^s}<\infty$ is satisfies by the usual power type moduli of continuity $\omega(t)=t^\delta$, which are available in most applications.
\end{remark}

\begin{remark}\label{rem:classicalT1}
In the original $T(1)$ theorem \cite{DJ:T1}, the boundedness \eqref{it:T1bd} in Theorem \ref{thm:T1} is characterised just by $T(1)$ conditions like \eqref{it:T1T1} and a weak boundedness property like \eqref{it:T1wbp}; while the original formulation in terms of smooth test functions is slightly different, dyadic versions involving conditions just as written in  \eqref{it:T1T1} and \eqref{it:T1wbp} can be found e.g. in \cite{BCR,Figiel:90}; see also the recent account in \cite[Chapter 12]{HNVW3}. Thus conditions \eqref{it:T1swbp} and \eqref{it:T1contcont} or \eqref{it:T1measmeas} would seem to be redundant. However, there is a tricky detail here. Under assumptions \eqref{it:T1T1} and \eqref{it:T1wbp}, one can extend a bilinear form from the smaller test function space $\mathscr I_{\mathscr D}:=\lspan\{1_Q:Q\in\mathscr D\}\subsetneq\mathscr I$ to $L^p(\R^d)$. But, if $\tau$ was defined on $\mathscr I$ to begin with, we then have two competing definitions of the bilinear form on $\mathscr I\subseteq L^p(\R^d)$: the original one, and the one obtained by first restricting to $\mathscr I_{\mathscr D}$ and then extending to $L^p(\R^d)$. To check compatibility, some additional continuity of the original $\tau$ on the $\mathscr I$ would seem to be necessary, and one is very soon back to asking conditions like \eqref{it:T1swbp} and \eqref{it:T1contcont}.

Moreover, while assumptions weaker than \eqref{it:T1cont} or \eqref{it:T1meas} allow one to write down representation formulas {\em very much like} \eqref{eq:DRTmain}, and quite sufficient to extend $\tau$ from some test function space to $L^p(\R^d)$ (see \cite[Chapter 12]{HNVW3}), in order to make sense of \eqref{eq:DRTmain} {\em exactly as written} (in particular, averaging directly over a sum of complete shift operators, rather than taking a limit of averages over truncated shifts), the stated assumptions would seem to be close to minimal. In particular, the measurability \eqref{it:T1measmeas} is basically what is needed to make the expression inside the expectation $\E$ in \eqref{eq:DRTmain} measurable on the probability space over which the expectation is taken.
\end{remark}

\begin{remark}\label{rem:swbp}
It is evident that the weak boundedness property follows from the semi-weak boundedness property with $\Norm{\tau}{wbp}\leq\Norm{\tau}{swbp}$. The reason why we have formulated these conditions separately is that, of the two constants, only $\Norm{\tau}{wbp}$ enters quantitatively into the estimate of $\Norm{T}{\bddlin(L^p(\R^d))}$, while the finiteness of $\Norm{\tau}{swbp}$ is only used qualitatively to guarantee a certain convergence.

It is equally evident that both these conditions are necessary for \eqref{it:T1bd}; indeed,
\begin{equation*}
  \abs{\tau(1_R,1_S)}
  =\abs{\pair{T1_R}{1_S}}
  \leq\Norm{T}{\bddlin(L^p(\R^d))}\Norm{1_R}{p}\Norm{1_S}{p'}
  \leq\Norm{T}{\bddlin(L^p(\R^d))}\abs{Q}^{\frac1p+\frac{1}{p'}}.
\end{equation*}

The necessity of \eqref{it:T1contcont} follows from the continuity of translations on $L^p(\R^d)$ and $L^{p'}(\R^d)$: thus $z\mapsto 1_{Q+z}=1_Q(\cdot-z)$ is continuous from $\R^d$ to either $L^p(\R^d)$ or $L^{p'}(\R^d)$, and thus $z\mapsto\tau(1_{Q+z},1_{Q+z})=\pair{T1_{Q+z}}{1_{Q+z}}$ is continuous by the continuity of the operator $T$ and of the duality between $L^p(\R^d)$ and $L^{p'}(\R^d)$.

The necessity of \eqref{it:T1T1} is well-known from the classical theory; cf. \cite{DJ:T1}.

Since measurability is weaker than continuity, it is clear that \eqref{it:T1meas} follows from \eqref{it:T1cont}. The core of the theorem then consists of proving that \eqref{it:T1meas} implies \eqref{it:T1bd}. The representation formula \eqref{eq:DRTmain} will be established as a byproduct of this proof.
\end{remark}

\section{The BCR algorithm with an error term}\label{sec:BCR}

We will now perform a variant of the decomposition of a bilinear form $\tau(f,g)$, commonly known as the BCR algorithm after \cite{BCR}, although it would actually seem to go back to \cite{Figiel:90}.

The starting point is a sequence of ``approximate identities'' $(E_i)_{i\in\Z}$. Later on, we will be concerned with the case that
\begin{equation*}
  E_i f=\sum_{Q\in\mathscr D_i}E_Q f,\quad E_Q f:=1_Q\fint_Q f,
\end{equation*}
is the conditional expectation with respect to the dyadic cubes $\mathscr D_i$ of side-length $2^{-i}$, but the initial algebra is valid for any sequence of operators. We also denote
\begin{equation*}
  D_i:=E_{i+1}-E_i,\quad F_i:=I-E_i.
\end{equation*}

\begin{lemma}\label{lem:BCR}
With the notation just introduced, for all $a<b$, we have the formula
\begin{equation}\label{eq:BCRwithError}
  \tau(f,g) = \tau_{a,b}(f,g)+\mathscr E_{a,b}(f,g),
\end{equation}
with the main term
\begin{equation*}
  \tau_{a,b}(f,g):=\sum_{i=a}^{b-1}\Big(\tau(D_i f,D_i g) +\tau(D_i f,E_i g)  +\tau(E_i f,D_i g)\Big)
\end{equation*}
and the error term
\begin{equation*}
   \mathscr E_{a,b}(f,g)=\tau(E_a f,E_a g)+\tau(F_b f,g)+\tau(E_b f,F_b g).
\end{equation*}
\end{lemma}

\begin{proof}
It is immediate from the definitions that
\begin{equation}\label{eq:EDF}
  I=E_b+F_b=E_a+\sum_{i=a}^{b-1}(E_{i+1}-E_i)+F_b=E_a+\sum_{i=a}^{b-1}D_i+F_b.
\end{equation}
We will then apply \eqref{eq:EDF} to both arguments $f$ and $g$ of the bilinear form $\tau(f,g)$. 
This leads to the expansion
\begin{equation}\label{eq:9terms}
\begin{split}
  \tau(f,g) &= \tau\Big(E_a f+\sum_{i=a}^{b-1} D_i f+F_b f,E_a g+\sum_{j=a}^{b-1} D_j g+ F_b g\Big) \\
  &=\tau(E_a f,E_a g)+\tau(E_a f,\sum_{j=a}^{b-1} D_j g)+\tau(E_a f,F_b g) \\
  &\qquad+\tau(\sum_{i=a}^{b-1} D_i f,E_a g)+\tau(\sum_{i=a}^{b-1} D_i f,\sum_{j=a}^{b-1} D_j g)+\tau(\sum_{i=a}^{b-1} D_i f,F_b g) \\
  &\qquad+\tau(F_b f,E_a g)+\tau(F_b f,\sum_{j=a}^{b-1} D_j g)+\tau(F_b f,F_b g).
\end{split}
\end{equation}
The middle term involving the double sum can be regrouped according to whether $j=i$, $j<i$, or $j>i$ as follows:
\begin{equation}\label{eq:5terms}
\begin{split}
  &\tau(\sum_{i=a}^{b-1} D_i f,\sum_{j=a}^{b-1} D_j g) \\
  &=\sum_{i=a}^{b-1}\tau(D_i f,D_i g) +\sum_{i=a}^{b-1}\sum_{j=a}^{i-1}\tau(D_i f,D_j g)  +\sum_{j=a}^{b-1}\sum_{i=a}^{j-1}\tau(D_i f,D_j g) \\
  &=\sum_{i=a}^{b-1}\tau(D_i f,D_i g) +\sum_{i=a}^{b-1}\tau(D_i f,E_i g-E_a g)  +\sum_{j=a}^{b-1}\tau(E_j f-E_a f,D_j g) \\
  &=\sum_{i=a}^{b-1}\Big(\tau(D_i f,D_i g) +\tau(D_i f,E_i g)  +\sum_{j=a}^{b-1}\tau(E_i f,D_j g)\Big) \\
  &\qquad   -\sum_{i=a}^{b-1}\tau(D_i f,E_a g)-\sum_{i=a}^{b-1}\tau(E_a f,D_ig).
\end{split}
\end{equation}
The first three of \eqref{eq:5terms} agree with the first three terms \eqref{eq:BCRwithError}, and the two negative terms of \eqref{eq:5terms} cancel with two of the nine terms on the right of \eqref{eq:9terms}. 
Finally, observing that
\begin{equation*}
  \tau(F_b f,E_a g)+\tau(F_b f,\sum_{j=a}^{b-1} D_j g)+\tau(F_b f,F_b g)=\tau(F_b f,g),
\end{equation*}
and similarly with the roles of $f$ and $g$ interchanged, we see that these remaining terms of \eqref{eq:9terms} combine to give
\begin{equation*}
\begin{split}
  \tau(E_a f,E_a g) &+\tau(F_b f,g)+\tau(f,F_b g)-\tau(F_b f,F_b g) \\
  &=\tau(E_a f,E_a g)+\tau(F_b f,g)+\tau(E_b f,F_b g),
\end{split}
\end{equation*}
which is precisely the claimed error term in \eqref{eq:BCRwithError}.
\end{proof}

\section{The error term}\label{sec:error}

We next show that the error term in the BCR algorithm vanishes in the limit. The first part is standard, but we recall the short argument for completeness.

\begin{lemma}\label{lem:EafEag}
For $f,g\in L^\infty_c(\R^d)$, we have
\begin{equation*}
   \tau(E_a f,E_a g)\to 0\quad\text{as}\quad a\to-\infty.
\end{equation*}
\end{lemma}

\begin{proof}
Suppose that $f,g\in L^\infty_c(\R^d)$. For $a$ negative enough (thus $2^{-a}$ large enough), the supports of $f$ and $g$ intersect at most $2^d$ cubes $Q\in\mathscr D_a$. Thus
\begin{equation*}
  \tau(E_a f,E_a g)
  =\sum_{P,Q\in\mathscr D_a}\ave{f}_P\tau(1_P,1_Q)\ave{g}_Q,
\end{equation*}
where the summation runs over the said $2^d$ cubes $P$ and $Q$ only.

The assumed weak boundedness property guarantees that $\abs{\tau(1_P,1_P)}\leq\Norm{\tau}{wbp}\abs{P}$ for each cube $P$. The similar estimate $\abs{\tau(1_P,1_Q)}\lesssim c_K\abs{P}$ for disjoint cubes of equal size follows from the standard kernel estimate $\abs{K(x,y)}\leq c_K\abs{x-y}^{-d}$ \cite[Lemma 12.4.2]{HNVW3}.
It is clear that $\abs{\ave{f}_P}\leq\abs{P}^{-1}\Norm{f}{\infty}$ and likewise for $g$. Hence
\begin{equation*}
\begin{split}
  \abs{\tau(E_a f,E_a g)}
  &\lesssim (\Norm{\tau}{wbp}+c_K)\Norm{f}{\infty}\Norm{g}{\infty}\abs{P}^{-1-1+1} \\
  &=(\Norm{\tau}{wbp}+c_K)\Norm{f}{\infty}\Norm{g}{\infty} 2^{da}\to 0
\end{split}
\end{equation*}
as $a\to-\infty$.
\end{proof}

The next estimate is slightly more exotic. It trivialises for $f\in\mathscr I_{\mathscr D}$, since then $F_b f=0$ for sufficiently large $b$. For $f,g\in\mathscr I$, it depends on the semi-weak boundedness property.

\begin{lemma}\label{lem:Fbfg}
For $f,g\in\mathscr I$, we have
\begin{equation*}
   \tau(F_b f,g)\to 0\quad\text{as}\quad b\to\infty.
\end{equation*}
\end{lemma}

\begin{proof}
We note that
\begin{equation*}
  F_b f=\sum_{Q\in\mathscr D_b}F_Q f,\quad F_Q f:=1_Q(f-\ave{f}_Q).
\end{equation*}
Then
\begin{equation*}
  \tau(F_b f,g)
  =\sum_{Q\in\mathscr D_b}\tau(F_Q f,g)
  =\sum_{Q\in\mathscr D_b}\tau(F_Q f,1_{3Q}g)+\sum_{Q\in\mathscr D_b}\tau(F_Q f,1_{(3Q)^c}g).
\end{equation*}
We first consider the second sum. We use the cancellation in $F_Q f$ to write
\begin{equation*}
\begin{split}
  \tau(F_Q f,1_{(3Q)^c}g)
  &=\iint K(x,y)F_Q f(y)1_{(3Q)^c}(x)g(x)\ud x\ud y \\
  &=\iint [K(x,y)-K(x,z_Q)]F_Q f(y)1_{(3Q)^c}(x)g(x)\ud x\ud y \\
\end{split}
\end{equation*}
and then kernel bounds to estimate
\begin{equation*}
\begin{split}
  \abs{\tau(F_Q f,1_{(3Q)^c}g)}
  &\lesssim\iint \omega(\frac{\frac12\ell(Q)}{\abs{x-z_Q}})\frac{1}{\abs{x-z_Q}^d}\abs{F_Qf(y)}1_{(3Q)^c}(x)\abs{g(x)}\ud x\ud y \\
  &\lesssim\Norm{F_Q f}{1}\Norm{\omega}{\operatorname{Dini}}\inf_{z\in Q} Mg(z)
  \leq\Norm{\omega}{\operatorname{Dini}}\int_Q\abs{F_b f(x)}Mg(x)\ud x,
\end{split}
\end{equation*}
where
\begin{equation*}
  \Norm{\omega}{\operatorname{Dini}}:=\int_0^1\omega(u)\frac{\ud u}{u}.
\end{equation*}
Summing over $Q\in\mathscr D_b$, we find that
\begin{equation}\label{eq:sumF}
   \sum_{Q\in\mathscr D_b}\abs{\tau(F_Q f,1_{(3Q)^c}g)}
   \lesssim\Norm{\omega}{\operatorname{Dini}}\int_{\R^d}\abs{F_b f(x)}Mg(x)\ud x.
\end{equation}
Here
\begin{equation*}
\begin{split}
  \abs{F_b f(x)} &=\abs{f(x)-\ave{f}_Q},\qquad x\in Q\in\mathscr D_b, \\ 
  &\to 0\qquad\text{as }b\to\infty
\end{split}
\end{equation*}
at almost every $x\in\R^d$ by Lebesgue's differentiation (or martingale convergence) theorem. On the other hand, we have $F_b f(x)\leq\abs{f(x)}+Mf(x)\leq 2Mf(x)$, and $Mf\cdot Mg\in L^1(\R^d)$ for $f\in L^p(\R^d)$ and $g\in L^{p'}(\R^d)$. Thus
\begin{equation*}
  \int_{\R^d}\abs{F_b f(x)}Mg(x)\ud x\to 0\quad\text{as }b\to\infty
\end{equation*}
by dominated convergence.

We then turn to the local part
\begin{equation*}
  \sum_{Q\in\mathscr D_b}\tau(F_Q f,1_{3Q}g),\quad
  \tau(F_Q f,1_{3Q}g)
  =\tau(F_Q f,1_{Q}g)+\tau(F_Q f,1_{3Q\setminus Q}g).
\end{equation*}
We will now make use of the stronger assumption that $f$ and $g$ are linear combinations of indicators of rectangles. By linearity, we may assume without loss of generality that $f=1_R$ and $g=1_S$ for some rectangles $R$ and $S$. We note that
\begin{equation}\label{eq:FQ1R}
  F_Q f=F_Q 1_R=1_{Q\cap R}-\frac{\abs{Q\cap R}}{\abs{Q}}1_Q
\end{equation}
is non-zero only if $Q$ intersects both $R$ and $R^c$, and thus $\partial R$.

Now
\begin{equation*}
  \tau(F_Q f,1_{3Q\setminus Q}g)=\sum_{\substack{Q'\in\mathscr D_b \\ Q'\subseteq 3Q\setminus Q}}\tau(F_Q f,1_{Q'}g),
\end{equation*}
and it follows from the kernel bound $\abs{K(x,y)}\leq c_K\abs{x-y}^{-d}$ that
\begin{equation}\label{eq:semilocal}
  \abs{\tau(F_Q f,1_{Q'}g)}\leq c_K\Norm{F_Q f}{\infty}\Norm{g}{\infty}\iint_{Q\times Q'}\frac{\ud x\ud y}{\abs{x-y}^d}\lesssim c_K\abs{Q}.
\end{equation}
Thus
\begin{equation}\label{eq:bdryComp}
\begin{split}
  \sum_{Q\in\mathscr D_b}\abs{\tau(F_Q f,1_{3Q\setminus Q}g)}
  &\lesssim c_K\sum_{\substack{ Q\in\mathscr D_b \\ Q\cap(\partial R)\neq\varnothing}}\abs{Q}  \\
  & \leq c_K\abs{\{y\in\R^d:\dist(y,\partial R)<2^{-b}\}} \\
  &\lesssim c_K 2^{-b}\mathcal H^{d-1}(\partial R)\to 0\quad\text{as }b\to\infty.
\end{split}
\end{equation}

Finally, referring to \eqref{eq:FQ1R},
\begin{equation*}
  \tau(F_Q f,1_Q g)
  =\tau(1_{Q\cap R},1_{Q\cap S})-\frac{\abs{Q\cap R}}{\abs{Q}}\tau(1_Q,1_{Q\cap S}).
\end{equation*}
Here we make use of the semi-weak boundedness property: Since $Q\cap R$, $Q\cap S$, and $Q$ itself are all rectangles contained in $Q$, we have
\begin{equation}\label{eq:verylocal}
  \abs{\tau(1_{Q\cap R},1_{Q\cap S})}\leq\Norm{\tau}{swbp}\abs{Q},\quad\abs{\tau(1_{Q},1_{Q\cap S})}\leq\Norm{\tau}{swbp}\abs{Q},
\end{equation}
and hence
\begin{equation*}
  \abs{\tau(F_Q f,1_Q g)}\lesssim \Norm{\tau}{swbp}\abs{Q}.
\end{equation*}
Summing over $Q$, it follows as in \eqref{eq:bdryComp} that
\begin{equation*}
\begin{split}
    \sum_{Q\in\mathscr D_b}\abs{\tau(F_Q f,1_{Q}g)}
  &\lesssim\Norm{\tau}{swbp}\sum_{\substack{ Q\in\mathscr D_b \\ Q\cap(\partial R)\neq\varnothing}}\abs{Q} \\
  &\lesssim \Norm{\tau}{swbp}2^{-b}\mathcal H^{d-1}(\partial R)\to 0\quad\text{as }b\to\infty.
\end{split}
\end{equation*}
\end{proof}

\begin{lemma}\label{lem:EbfFbg}
For $f,g\in\mathscr I$, we have
\begin{equation*}
  \tau(E_b f,F_b g)\to 0\quad\text{as}\quad b\to\infty.
\end{equation*}
\end{lemma}

\begin{proof}
Since the assumptions on $\tau$ are symmetric with respect to the two arguments, we may as well consider $\tau(F_b f, E_b g)$. This is like $\tau(F_b f,g)$ of Lemma \ref{lem:Fbfg} but with $E_b g$ in place of $g$, so it remains to explain the modification caused by this replacement.

In place of \eqref{eq:sumF}, we obtain
\begin{equation}\label{eq:sumF2}
   \sum_{Q\in\mathscr D_b}\abs{\tau(F_Q g, 1_{(3Q)^c}E_b g,)}
   \lesssim\Norm{\omega}{\operatorname{Dini}}\int_{\R^d}\abs{F_b f(x)}M(E_b g)(x)\ud x.
\end{equation}
But $\abs{E_b g}\leq Mf$, and $M(E_b f)\leq M(Mg)$ is still dominated in norm by the original $g$, so the convergence of this term is concluded, as before, by the pointwise convergence $F_b f(x)\to 0$ followed by dominated convergence. (Actually, a slightly more careful inspection of the considerations leading to \eqref{eq:sumF} shows that we can get \eqref{eq:sumF2} with simply $Mg$ in place of $M(E_b g)$.)

As for the semi-local part \eqref{eq:semilocal}, the replacement of $g$ by $E_b g$ leads to the same bound as $\Norm{E_b g}{\infty}\leq\Norm{g}{\infty}$.

In the very local estimate \eqref{eq:verylocal}, the replacement of $g=1_S$ by $E_b g$ leads to the replacement of $1_Q g=1_{Q\cap S}$ by $\frac{\abs{Q\cap S}}{\abs{Q}}1_Q$, and the analogues of \eqref{eq:verylocal} follow from the semi-weak boundedness property as before.
\end{proof}

\begin{corollary}\label{cor:BCRerrorToZero}
For $f,g\in\mathscr I$, we have
\begin{equation*}
  \mathscr E_{a,b}(f,g)\to 0\quad\text{as}\quad a\to-\infty,\quad b\to\infty.
\end{equation*}
\end{corollary}

\begin{proof}
The quantity $\mathscr E_{a,b}(f,g)$ consists of three terms. The convergence of each term to zero is handled in Lemmas \ref{lem:EafEag}, \ref{lem:Fbfg} and \ref{lem:EbfFbg}, respectively.
\end{proof}

\section{The main term}\label{sec:main}

Turning to the main term in Lemma \ref{lem:BCR}, we will now express it in the following form, where we denote $\mathscr D_{[a,b)}:=\bigcup_{i\in[a,b)}\mathscr D_i$.

\begin{lemma}\label{lem:mainterm}
\begin{equation*}
  \tau_{a,b}(f,g)
  =\tau_{a,b}^{\operatorname{diag}}(f,g)+\tau_{a,b}^{\operatorname{off}}(f,g)
\end{equation*}
with the diagonal part
\begin{equation*}
  \tau_{a,b}^{\operatorname{diag}}(f,g)
  :=\sum_{P\in\mathscr D_{[a,b)}}\Big(\tau(D_P f,D_P g)+\ave{f}_P\tau(1,D_P g)+\tau(D_P f,1)\ave{g}_P\Big)
\end{equation*}
and the off-diagonal part
\begin{equation*}
\begin{split}
  \tau_{a,b}^{\operatorname{off}}(f,g)
  &:=\sum_{\gamma\in\{0,1\}^2\setminus\{0\}}\sum_{k=2}^\infty \tau_{a,b}^{(\gamma,k)}(f,g), \\
  \tau_{a,b}^{(\gamma,k)}(f,g)
  &:=\sum_{\substack{P,Q\in\mathscr D_{[a,b)} \\ 2^{k-3}<\abs{z_P-z_Q}/\ell(P)\leq 2^{k-2}}}
   \tau(D_{P,Q}^{\gamma_1} f,D_{Q,P}^{\gamma_2}g)
\end{split}
\end{equation*}
where
\begin{equation*}
  D_{P,Q}^1 f:=D_P f,\qquad D_{P,Q}^0 f:=D_{P,Q} f:=(\ave{f}_P-\ave{f}_Q)1_P.
\end{equation*}
\end{lemma}

\begin{remark}
Bi-linear forms with diagonal part only (i.e., with $\tau_{a,b}^{\operatorname{off}}\equiv 0$) are studied in \cite{AHMTT} under the name ``perfect dyadic Calder\'on--Zygmund operators''.
\end{remark}

\begin{proof}[Proof of Lemma \ref{lem:mainterm}]
Referring to Lemma \ref{lem:BCR} for the definition of $\tau_{a,b}(f,g)$, we begin by writing out and decomposing one of the terms as follows:
\begin{equation}\label{eq:DfEg}
\begin{split}
  \tau(D_i f,E_i g)
  &=\sum_{P,Q\in\mathscr D_i}\tau(D_P f,E_Q g) 
  =\sum_{P,Q\in\mathscr D_i}\tau(D_P f,1_Q)\ave{g}_Q \\
  &=\sum_{P,Q\in\mathscr D_i}\tau(D_P f,1_Q)\Big((\ave{g}_Q-\ave{g}_P)+\ave{g}_P\Big).
\end{split}
\end{equation}
In the original form with $\tau(D_P f,E_Q g)$, this term is non-zero only if $P$ meets the support of $f$ and $Q$ meets the support of $g$; thus both cubes range over finitely many choices only, and there are no issues of convergence. We would like to split the sum into two at the plus sign on the right, but this needs some care: for $\tau(D_P f,1_Q)\ave{g}_P$ to be nonzero, the cube $P$ must meet the supports of both $f$ and $g$, and is hence restricted to a finite set, but the cube $Q\in\mathscr D_i$ may be arbitrary, and hence we need to take care of the convergence of an infinite series to make such a splitting.

We will shortly recall a well-known argument showing that, under standard kernel assumptions, the series
\begin{equation}\label{eq:tauDP1}
  \sum_{Q\in\mathscr D_i}\tau(D_P f,1_Q)=:\tau(D_P f,1)
\end{equation}
converges absolutely (and uniformly with respect to the choice of the system of dyadic cubes $\mathscr D_i$ of sidelength $2^{-i}$). Taking this for granted for the moment, it follows that also
\begin{equation}\label{eq:tauDP1g}
  \sum_{P,Q\in\mathscr D_i}\tau(D_P f,1_Q)\ave{g}_P
  =\sum_{P\in\mathscr D_i}\Big(\sum_{Q\in\mathscr D_i}\tau(D_P f,1_Q)\Big)\ave{g}_P
  =\sum_{P\in\mathscr D_i}\tau(D_P f,1)\ave{g}_P
\end{equation}
converges absolutely as a sum of finitely many (corresponding to those $P\in\mathscr D_i$ that intersect the bounded support of $g$) absolutely convergent series. Since
\begin{equation*}
  \sum_{P,Q\in\mathscr D_i}\tau(D_P f,1_Q)(\ave{g}_Q-\ave{g}_P)
\end{equation*}
differs from \eqref{eq:tauDP1g} by the finitely nonzero sum $\sum_{P,Q\in\mathscr D_i}\tau(D_P f,1_Q)\ave{g}_Q$ only, it must also converge absolutely.

To see the absolute convergence of \eqref{eq:tauDP1}, for $Q\subset(3P)^c$, we have
\begin{equation*}
\begin{split}
    \tau(D_P f,1_Q)
    &=\iint K(x,y)D_Pf(y)1_Q(x)\ud y\ud x \\    
    &=\iint [K(x,y)-K(x,z_P)]D_Pf(y)1_Q(x)\ud y\ud x,
\end{split}
\end{equation*}
hence
\begin{equation*}
\begin{split}
  \abs{\tau(D_P f,1_Q)} &\lesssim\iint \omega\Big(\frac{\frac12\ell(P)}{\abs{x-z_P}}\Big)\frac{1}{\abs{x-z_P}^d}\abs{D_Pf(y)}1_Q(x)\ud y\ud x \\
  &=\Norm{D_P f}{1}\int_Q \omega\Big(\frac{\frac12\ell(P)}{\abs{x-z_P}}\Big)\frac{1}{\abs{x-z_P}^d}\ud x
\end{split}
\end{equation*}
and thus
\begin{equation*}
\begin{split}
  \sum_{\substack{Q\in\mathscr D_i \\ Q\subset(3P)^c}} \abs{\tau(D_P f,1_Q)}
  &\lesssim\Norm{D_P f}{1}\int_{(3P)^c} \omega\Big(\frac{\frac12\ell(P)}{\abs{x-z_P}}\Big)\frac{1}{\abs{x-z_P}^d}\ud x \\
  &\lesssim\Norm{D_P f}{1}\Norm{\omega}{\operatorname{Dini}}.
\end{split}
\end{equation*}
While the convergence of the remaining finite sum over $Q\subset 3P$ is obvious, we also record the following explicit bounds:
For $Q\subset (3P)\setminus P$, we have
\begin{equation*}
   \abs{\tau(D_P f,1_Q)}
   \leq \iint\frac{c_K}{\abs{x-y}^d}\abs{D_P f(y)}1_Q(x)\ud y\ud x
   \lesssim c_K\abs{P}\Norm{D_P f}{\infty}\lesssim C_K\Norm{D_P f}{1}.
\end{equation*}
For $Q=P$, denoting by $P_i$ its dyadic children, we have
\begin{equation*}
   \abs{\tau(D_P f,1_P)}
   \leq\sum_{i,j=1}^{2^d}\abs{\ave{D_P f}_{P_i}}\abs{\tau(1_{P_i},1_{P_j})},
\end{equation*}
where $\abs{\tau(1_{P_i},1_{P_j})}\lesssim c_K\abs{P_i}$ for $i\neq j$ by the kernel bound, and $\abs{\tau(1_{P_i},1_{P_i})}\leq\Norm{\tau}{wbp}\abs{P_i}$ for $i=j$ by definition of the weak boundedness property. Hence
\begin{equation*}
  \abs{\tau(D_P f,1_P)}\lesssim(c_K+\Norm{\tau}{wbp})\sum_{i=1}^{2^d}\abs{P_i}\abs{\ave{D_P f}_{P_i}}
  =(c_K+\Norm{\tau}{wbp})\Norm{D_P f}{1}.
\end{equation*}
A combination of the previous few estimates shows that
\begin{equation*}
    \sum_{Q\in\mathscr D_i } \abs{\tau(D_P f,1_Q)}\lesssim\Big(\Norm{\tau}{wbp}+c_K+\Norm{\omega}{\operatorname{Dini}}\Big)\Norm{D_P f}{1}.
\end{equation*}

Having dealt with this absolute converge, we return to manipulating the original series.
Subsituting \eqref{eq:tauDP1g} into \eqref{eq:DfEg}, we obtain
\begin{equation*}
  \tau(D_i f,E_i g)
  =\sum_{P,Q\in\mathscr D_i}\tau(D_P f,1_Q)(\ave{g}_Q-\ave{g}_P)+\sum_{P\in\mathscr D_i}\tau(D_P f,1)\ave{g}_P.
\end{equation*}
Similarly,
\begin{equation*}
  \tau(E_i f,D_i g)
  =\sum_{P,Q\in\mathscr D_i}(\ave{f}_P-\ave{f}_Q)\tau(1_P,D_Q g)+\sum_{Q\in\mathscr D_i}\ave{f}_Q\tau(1,D_Q g).
\end{equation*}
Using the notation introduced in the statement of the lemma,
we can combine these into
\begin{equation}\label{eq:extractParap}
\begin{split}
  &\tau(D_i f,D_i g)+\tau(D_i f,E_i g)+\tau(E_i f,D_i g) \\
  &=\sum_{P,Q\in\mathscr D_i}[\tau(D_P f,D_Q g)+\tau(D_{P,Q}f,D_Q g)+\tau(D_P f,D_{Q,P}g)] \\
  &\qquad+\sum_{P\in\mathscr D_i}[\ave{f}_P\tau(1,D_P g)+\tau(D_P f,1)\ave{g}_P].
\end{split}
\end{equation}
For $P=Q$, we have $D_{P,Q}=D_{Q,P}=0$, and hence
\begin{equation}\label{eq:extractHaarMult}
\begin{split}
  \sum_{P,Q\in\mathscr D_i} &[\tau(D_P f,D_Q g)+\tau(D_{P,Q}f,D_Q g)+\tau(D_P f,D_{Q,P}g)] \\
  &=\sum_{P\in\mathscr D_i}\tau(D_P f,D_P g)+\sum_{\substack{P,Q\in\mathscr D_i \\ P\neq Q}}\sum_{\gamma\in\{0,1\}^2\setminus\{0\}}
    \tau(D_{P,Q}^{\gamma_1}f,D_{Q,P}^{\gamma_2}g).
\end{split}
\end{equation}
When summed over $i\in[a,b)$, the second line on the right of \eqref{eq:extractParap} and the first term on the right of \eqref{eq:extractHaarMult} clearly combine to give the term $\tau_{a,b}^{\operatorname{diag}}(f,g)$ as asserted in the lemma. So it remains to check that last term of \eqref{eq:extractHaarMult} agrees with $\tau_{a,b}^{\operatorname{off}}(f,g)$. But this is straightforward: Since $P,Q\in\mathscr D_i$ with $P\neq Q$, their centres $z_P,z_Q$ are separated by $\abs{z_P-z_Q}\geq\ell(P)$, and we simply partition the values of $\abs{z_P-z_Q}/\ell(P)\in[1,\infty)$ into the disjoint intervals $(2^{k-3},2^{k-2}]$, where $k\geq 2$. The absolute convergence of the series that we already observed justifies this reorganisation.
\end{proof}


\section{Questions of measurability related to randomisation}\label{sec:meas}

We now consider a choice of a random dyadic system
\begin{equation*}
  \mathscr D^\theta:=\{Q\dot+\theta:Q\in\mathscr D^0\},
\end{equation*}
where $\mathscr D^0:=\{2^{-i}([0,1)^d+m):i\in\Z,m\in\Z^d\}$ is the standard dyadic system, and
\begin{equation*}
  Q\dot+\theta:=Q+\sum_{i:2^{-i}<\ell(Q)}2^{-i}\theta_i
\end{equation*}
for $\theta=(\theta_i)_{i\in\Z}\in\Theta:=(\{0,1\}^d)^{\Z}$. We say that $R=Q\dot+\theta\in\mathscr D^{\theta}$ is $k$-good, if $R\subset\frac12 R^{(\theta,k)}$, where $R^{(\theta,k)}$ is the $k$th dyadic ancestor of $R$ in the system $\mathscr D^{\theta}$. If $\Theta$ is equipped with the canonical product probability $\P$, one find that, for all $Q\in\mathscr D^0$,
\begin{equation*}
  \P(\{\theta:Q\dot+\theta\text{ is $(k,\theta)$-good}\})=2^{-d}
\end{equation*}
and the spatial position and the $(k,\theta)$-goodness of $Q\dot+\theta$ are independent \cite[Lemma 12.3.33]{HNVW3}.

In order to compute expectations over sums like those in Lemma \ref{lem:mainterm}, where $\mathscr D=\mathscr D^\theta$ is a random dyadic system, some issues of measurability need to be addressed. 

\begin{lemma}\label{lem:measmain}
For $\gamma\in\{0,1\}^2\setminus\{0\}$, the function
$\theta\in\Theta\mapsto\tau(D_{P\dot+\theta,Q\dot+\theta}^{\gamma_1}f,D_{Q\dot+\theta,P\dot+\theta}^{\gamma_2} g)$ is measurable.
\end{lemma}

\begin{proof}
Since the mapping $\theta\in\Theta\mapsto\sum_{i:2^{-i}<\ell(P)}2^{-i}\theta_i\in[0,\ell(P)]^d$ is measurable, it is enough to verify that
\begin{equation*}
  u\in[0,\ell(P)]^d\mapsto \tau(D_{P+u,Q+u}^{\gamma_1}f,D_{Q+u,P+u}^{\gamma_2} g)
\end{equation*}
is measurable. 
Let first $\gamma=(1,1)$, i.e., $D_{P+u,Q+u}^{\gamma_1}=D_{P+u}$ and $D_{Q+u,P+u}^{\gamma_2}=D_{Q+u}$. We note that any $D_R$ can be written in terms of the Haar functions as
\begin{equation*}
  D_R f=\sum_{\alpha\in\{0,1\}^d\setminus\{0\}}\pair{f}{h_R^\alpha}h_R^\alpha,\qquad
  h_R^\alpha(x):=\prod_{i=1}^d h_{R_i}^{\alpha_i}(x_i),
\end{equation*}
where, denoting by $I_{\rm left}$ and $I_{\rm right}$ the left and right halves of an interval $I\subset\R$,
\begin{equation*}
  h_I^0:=\abs{I}^{-1/2}1_I,\quad h_I^1:=\abs{I}^{-1/2}(1_{I_{\rm left}}-1_{I_{\rm right}}).
\end{equation*}
It is immediate from these formulas that
\begin{equation*}
  h_{Q+z}^\alpha=h_Q^\alpha(\cdot-z),
\end{equation*}
i.e., a Haar function on a translated cube is simply the corresponding translate of the Haar function on the original cube.

Thus
\begin{equation*}
  \tau(D_{P+u}f,D_{Q+u}g)
  =\sum_{\alpha,\beta\in\{0,1\}^d\setminus\{0\}}\pair{f}{h_{P+u}^\alpha}\tau(h_{P+u}^\alpha,h_{Q+u}^\beta)\pair{g}{h_{Q+u}^\beta},
\end{equation*}
where
\begin{equation*}
  \pair{f}{h_{P+u}^\alpha}
  =\pair{f}{h_P^\alpha(\cdot-u)}=\pair{f(\cdot+u)}{h_P^\alpha}
\end{equation*}
and similarly for $g$. Since $u\mapsto f(\cdot+u)\in L^p(\R^d)$ is continuous, it is clear that these factors are continuous and hence measurable in $u$.

Concerning the factor involving the bilinear form, let first $P\neq Q$. In this case, the bilinear form is defined by the absolutely convergent integral
\begin{equation}\label{eq:tauhh}
\begin{split}
  \tau(h_{P+u}^\alpha,h_{Q+u}^\beta)
  &=\iint K(x,y)h_{P+u}^\alpha(y)h_{Q+u}^\beta(x)\ud x\ud y \\
  &=\iint K(x,y)h_{P}^\alpha(y-u)h_{Q}^\beta(x-u)\ud x\ud y \\
  &=\iint K(x+u,y+u)h_{P}^\alpha(y) h_{Q}^\beta(x) \ud x\ud y.
\end{split}
\end{equation}
The Calder\'on--Zygmund kernel $K$ is continuous away from the diagonal; thus $K(x+u,y+u)\to K(x+u_0,y+u_0)$ at every $x\neq y$ as $u\to u_0$. On the other hand, the integrand is pointwise dominated by $\abs{(x+u)-(y+u)}^{-d}=\abs{x-y}^{-d}$, which is integrable over $(x,y)\in Q\times P$ for disjoint cubes $P$ and $Q$, and hence the continuity, and thus measurability, of this factor follows from dominated convergence.

Let us then consider the case $P=Q$. We denote by $P_1$ the ``lower left quadrant'' of $P$. Then
\begin{equation*}
  h_{P}^\alpha
  =\sum_{\gamma\in\{0,\ell(P)/2\}} 1_{P_1+\gamma}\ave{h_{P}^\alpha}_{P_1+\gamma}
\end{equation*}
and hence
\begin{equation*}
  h_{P+u}^\alpha=h_P^\alpha(\cdot-u)
 =\sum_{\gamma\in\{0,\ell(P)/2\}} 1_{P_1+\gamma}(\cdot-u)\ave{h_{P}^\alpha}_{P_1+\gamma}.
\end{equation*}
It follows that
\begin{equation*}
  \tau(h_{P+u}^\alpha,h_{P+u}^\beta)
  =\sum_{\gamma,\delta\in\{0,\ell(P)/2\}^d}\ave{h_{P}^\alpha}_{P_1+\gamma}\ave{h_P^\beta}_{P_1+\delta}
    \tau(1_{P_1+\gamma}(\cdot-u),1_{P_1+\delta(\cdot-u)}),
\end{equation*}
and it remains to check the measurability of the last factor. For $\gamma\neq\delta$, we can use the kernel representation as in \eqref{eq:tauhh} to see that
\begin{equation*}
  \tau(1_{P_1+\gamma}(\cdot-u),1_{P_1+\delta(\cdot-u)})
  =\iint K(x+u,y+u)1_{P_1+\gamma}(y)1_{P_1+\delta}(x)\ud x\ud y,
\end{equation*}
whose continuity and hence measurability follows from the continuity of $u\mapsto K(x+u,y+u)$ and dominated convergence. For $\gamma=\delta$, we have
\begin{equation*}
  \tau(1_{P_1+\gamma}(\cdot-u),1_{P_1+\gamma(\cdot-u)})
  =\tau(1_{P_1+\gamma+u},1_{P_1+\gamma+u}),
\end{equation*}
which is measurable by the assumed weak measurability property \eqref{it:T1measmeas} of Theorem \ref{thm:T1} (with the cube $P_1+\gamma$ in place of $Q$).

This completes the proof of the measurability of $u\mapsto\tau(D_{P+u}f,D_{Q+u}g)$, and we turn to the case $\gamma=(1,0)$, i.e., $D_{P+u,Q+u}^{\gamma_1} f=D_{P+u}f$ and $D_{Q+u,P+u}^{\gamma_2} g=(\ave{g}_{Q+u}-\ave{g}_{P+u})1_{Q+u}$. Thus we are dealing with
\begin{equation*}
\begin{split}
  \tau(D_{P+u}f, &(\ave{f}_{Q+u}-\ave{f}_{P+u})1_{Q+u}) \\
  &=\sum_{\alpha\in\{0,1\}^d\setminus\{0\}}\pair{f}{h_{P+u}^\alpha}\tau(h_{P+u}^\alpha,1_{Q+u})(\ave{g}_{Q+u}-\ave{g}_{P+u}).
\end{split}
\end{equation*}
The first factor is continuous in $u$ exactly as before, and the continuity of the last factor follows similarly from the continuity of translations of $L^{p'}(\R^d)$. Moreover, this last factor 
 vanishes for $Q=P$, so it now suffices to deal with $Q\neq P$, where the kernel representation is available for the middle factor. As in \eqref{eq:tauhh}, this can be written as
 \begin{equation*}
  \tau(h_{P+u}^\alpha,1_{Q+u})
  =\iint K(x+u,y+u)h_P^\alpha(y)1_Q(x)\ud x\ud y,
\end{equation*}
and the continuity follows in exactly the same way as before. Finally, the case of $\gamma=(0,1)$, i.e., of $\tau(D_{P+u,Q+u}f,D_Q g)$, is entirely symmetric.
\end{proof}

\begin{lemma}\label{lem:measBMO}
$\theta\in\Theta\mapsto\tau(1,D_{P\cdot+\theta}g)$ and $\theta\in\Theta\mapsto\tau(D_{P\dot+\theta}f,1)$ are measurable.
\end{lemma}

\begin{proof}
By symmetry, it is enough to consider the first case.
By the $T(1)$ assumptions, we have $\tau(1,D_{P\cdot+\theta}g)=\pair{b_1}{D_{P\cdot+\theta}g}$.
As in the proof of Lemma \ref{lem:measmain}, we are reduced to considering the measurability of $u\in[0,\ell(P)]^d\mapsto \pair{b_1}{D_{P+u}g}$ and further, by expanding in terms of the Haar functions, that of
\begin{equation*}
 u\mapsto\pair{b_1}{h_{P+u}^\alpha}
 =\pair{b_1}{h_P^\alpha(\cdot-u)}.
\end{equation*}
This is, in fact, continuous, which follows from the assumption that $b_1\in\BMO(\R^d)$, which acts continuously on the Hardy space $H^1(\R^d)$, from the observation that $h_P^\alpha\in H^1(\R^d)$, and finally from the continuity of translations on $H^1(\R^d)$. Alternatively, one can give the following explicit estimate:
\begin{equation*}
\begin{split}
  \abs{\pair{b_1}{h_{P+u}^\alpha}-  \pair{b_1}{h_P^\alpha}}
  &=\abs{\pair{b_1}{h_{P+u}^\alpha-h_P^\alpha}} \\
  &=\abs{\pair{b_1-c}{h_{P+u}^\alpha-h_P^\alpha}}, \qquad c=\text{any constant}, \\
  &\leq\int \abs{b_1-c}\abs{h_{P+u}^\alpha-h_P^\alpha} \\
  &\leq\Big( \int_{3P}\abs{b_1-c}^2\Big)^{1/2}\Norm{h_{P+u}^\alpha-h_P^\alpha}{L^2(\R^d)} \\
  &\lesssim\abs{3P}^{1/2}\Norm{b_1}{\BMO(\R^d)}\Norm{h_{P}^\alpha(\cdot-u)-h_P^\alpha}{L^2(\R^d)},
\end{split}
\end{equation*}
where the last factor tends to $0$ as $u\to 0$ by the continuity of translations on $L^2(\R^d)$.
\end{proof}


Let us also recall the error term, where we now spell out the dependence on the particular $\mathscr D=\mathscr D^\theta$:
\begin{equation}\label{eq:errorAltForm}
\begin{split}
  \mathscr E_{a,b}^{\mathscr D}(f,g)
  &= \tau(E_a^{\mathscr D}f,E_a^{\mathscr D}g)+
    \tau(F_b^{\mathscr D}f,g)+   \tau(E_b^{\mathscr D}f,F_b^{\mathscr D}g)  \\
  &= \tau(E_a^{\mathscr D}f,E_a^{\mathscr D}g)+
    \tau((I-E_b^{\mathscr D})f,g)+    \tau(E_b^{\mathscr D}f,(I-E_b^{\mathscr D})g) \\
  &= \tau(E_a^{\mathscr D}f,E_a^{\mathscr D}g)+
   \tau(f,g)-\tau(E_b^{\mathscr D}f,E_b^{\mathscr D}g),
\end{split}
\end{equation}
where
\begin{equation*}
  \tau(E_a^{\mathscr D}f,E_a^{\mathscr D}g)
  =\sum_{P,Q\in\mathscr D_a}\tau(E_P f,E_Q g)
  =\sum_{P,Q\in\mathscr D_a}\ave{f}_P\tau(1_P,1_Q)\ave{g}_Q,
\end{equation*}
and the expression with $b$ in place of $a$ is of course analogous.

\begin{lemma}\label{lem:errormeas}
For any $P,Q\in\mathscr D^0$ of the same generation, the function $\theta\mapsto \tau(E_{P\dot+\theta}f,E_{Q\dot+\theta}g)$ is measurable.
\end{lemma}

\begin{proof}
As before, we are reduced to considering the measurability of
\begin{equation*}
 u\in[0,\ell(P)]^d\mapsto \tau(E_{P+u}f,E_{Q+u}g)=\ave{f}_{P+u}\tau(1_{P+u},1_{Q+u})\ave{g}_{Q+u}, 
\end{equation*}
where $u\mapsto \ave{f}_{P+u}=\ave{f(\cdot-u)}_P$ is continuous by the continuity of translations on $L^p(\R^d)$, and similarly with $u\mapsto\ave{g}_{Q+u}$. If $P\neq Q$ are cubes of the , then
\begin{equation*}
\begin{split}
  u\mapsto \tau(1_{P+u},1_{Q+u})
  &=\iint K(x,y)1_{P+u}(y)1_{Q+u}(x)\ud x\ud y \\
  &=\iint K(x+u,y+u)1_P(y)1_Q(x)\ud x\ud y
\end{split}
\end{equation*}
is continuous by the continuity of $K$ away from the diagonal, the uniform upper bound $\abs{K(x+u,y+u)}\leq c_K\abs{(x+u)-(y+u)}^{-d}=\abs{x-y}^{-d}$, the integrability of this function over $(x,y)\in Q\times P$, and dominated convergence. If $P=Q$, then $u\mapsto\tau(1_{P+u},1_{P+u})$ is measurable by the assumed weak measurability property.
\end{proof}

\begin{remark}
From Lemma \ref{lem:errormeas} and \eqref{eq:errorAltForm}, it follows that $\theta\mapsto\mathscr E_{a,b}^{\mathscr D^\theta}(f,g)$ is measurable. This would also follow from Lemmas \ref{lem:measmain} and \ref{lem:measBMO}, which guarantee the measurability of $\theta\mapsto\tau_{a,b}^{\mathscr D^\theta}(f,g)$, and the identity
\begin{equation*}
  \mathscr E_{a,b}^{\mathscr D^\theta}(f,g)=\tau(f,g)-\tau_{a,b}^{\mathscr D^\theta}(f,g),
\end{equation*}
where $\tau(f,g)$ is constant in $\theta$, and therefore trivially measurable. On the other hand, it is not obvious whether or not we have the separate measurability of the two terms $\tau(F_b^{\mathscr D^\theta}f,g)$ and $\tau(E_b^{\mathscr D^\theta}f,F_b^{\mathscr D^\theta}g)$ appearing in the first version of the expansion of $\mathscr E_{a,b}^{\mathscr D^\theta}(f,g)$ in \eqref{eq:errorAltForm}, although their sum is measurable according to the last identity in \eqref{eq:errorAltForm} and Lemma \ref{lem:errormeas}. This is not a problem, since we have no need to try to compute expectations of these individual terms.
\end{remark}


\section{Averaging over the off-diagonal part}\label{sec:off}

Referring back to the expansion of the main term in Lemma \ref{lem:mainterm}, the reason for introducing the randomisation is that it allows one to rework the off-diagonal terms into a more amenable form. (On the other hand, the diagonal part is good enough as it stands, for any fixed dyadic system $\mathscr D$. Averaging this part over a random choice of $\mathscr D$ does not bring any simplification, but does not hurt either.)

For the off-diagonal part, the key notion is the following:

\begin{definition}\label{def:good}
We say that $Q\in\mathscr D^\theta$ is $(k,\theta)$-good, if $Q\subset\frac12 Q^{(k,\theta)}$, where $Q^{(k,\theta)}$ is the $k$th dyadic ancestor of $Q$ in the dyadic system $\mathscr D^{\theta}$.
\end{definition}

\begin{remark}\label{rem:good}
This notion of goodness is closely related to, and inspired by, a notion introduced by Nazarov, Treil and Volberg \cite{NTV:Cauchy,NTV:Tb} and subsequently used by several authors. The main point is that a good cube should be sufficiently separated from the boundary of some bigger cube. The exact details of Definition \ref{def:good} are from \cite{GH:2018}.
\end{remark}

Several variants of the following basic lemma are well known; see \cite[Lemma 12.3.33]{HNVW3} for the present version.

\begin{lemma}\label{lem:good}
For $Q\in\mathscr D^0$ and $k\geq 2$,
\begin{enumerate}
  \item\label{it:indep} the random set $Q\dot+\theta$ and the event ``$Q\dot+\theta$ is $(\theta,k)$-good'' are independent,
  \item\label{it:2d} $\P(Q\dot+\theta\textup{ is }(\theta,k)\textup{-good})=2^{-d}$,
\end{enumerate}
where $\P$ refers to the canonical probability on $\Theta=(\{0,1\}^d)^{\Z}$.
\end{lemma}

With the help of these tools, we have:

\begin{lemma}\label{lem:shifts}
Denoting by $\E$ the expectation over $\theta\in\Theta$ with respect to the canonical probability on $\Theta=(\{0,1\}^d)^{\Z}$, the averages of the blocks $\tau_{a,b;\mathscr D^\theta}^{(\gamma,k)}(f,g)$ of  $\tau_{a,b;\mathscr D^\theta}^{\operatorname{off}}(f,g)$ in Lemma \ref{lem:mainterm} may be written as
\begin{equation*}
  \E\tau_{a,b;\mathscr D^\theta}^{(\gamma,k)}(f,g)
  =\omega(2^{-k})\E\mathfrak a_{a,b;\mathscr D^\theta}^{(\gamma,k)}(f,g),
\end{equation*}
where
\begin{equation*}
  \mathfrak a_{a,b;\mathscr D^\theta}^{(\gamma,k)}(f,g)
  =\sum_{S\in\mathscr D_{[a-k,b-k)}}a_{S}^{(\gamma,k)}(f,g)
\end{equation*}
and each $a_S^{(\gamma,k)}$ satisfies the size conditions
\begin{equation*}
   \abs{\mathfrak a_S^{(\gamma,k)}(f,g)}
   \lesssim\begin{cases}  \pair{E_S\abs{f}}{\abs{g}}, & \gamma=(1,1), \\
   \pair{(E_S+E_S^{(k)})\abs{f}}{\abs{g}}, & \text{else}, \end{cases}
\end{equation*}
and the cancellation conditions
\begin{equation*}
  \mathfrak a_S^{(\gamma,k)}(f,g)=\begin{cases} \mathfrak a_S^{(\gamma,k)}(D_S^{(k)}f,D_S^{(k)}g), & \gamma=(1,1), \\
     \mathfrak a_S^{(\gamma,k)}(D_S^{(k)}f,D_S^{[0,k)}g), & \gamma=(1,0), \\
      \mathfrak a_S^{(\gamma,k)}(D_S^{[0,k)}f,D_S^{(k)}g), & \gamma=(0,1), \end{cases}
\end{equation*}
where
\begin{equation*}
  D_S^{(k)}:=\sum_{\substack{P\in\mathscr D \\ P^{(k)}=S}}D_P,\quad
  E_S^{(k)}:=\sum_{\substack{P\in\mathscr D \\ P^{(k)}=S}}E_P,\quad
  D_S^{[0,k)}:=\sum_{j=0}^{k-1}D_S^{(j)}=E_S^{(k)}-E_S.
\end{equation*}
\end{lemma}

\begin{proof}
This can be found in \cite[Section 12.4.b]{HNVW3}, but we indicate some details for the reader's convenience.
Note that
\begin{equation*}
  \tau_{a,b;\mathscr D^\theta}^{(\gamma,k)}(f,g)
  =\sum_{\substack{P,Q\in\mathscr D_{[a,b)}^0 \\ 2^{k-3}<\abs{z_P-z_Q}/\ell(P)\leq 2^{k-2}}}
   \tau(D_{P\dot+\theta,Q\dot+\theta}^{\gamma_1} f,D_{Q\dot+\theta,P\dot+\theta}^{\gamma_2}g) \\
\end{equation*}
If $P\dot+\theta$ is $(k,\theta)$-good, so that $z_{P\dot+\theta}\in P+\theta\subset\frac12(P\dot+\theta)^{(k,\theta)}$, then
\begin{equation*}
  \operatorname{dist}(P\dot+\theta,\complement((P\dot+\theta)^{(k,\theta)}))\geq\frac14\ell((P\dot+\theta)^{(k,\theta)})=2^{k-2}\ell(P).
\end{equation*}
Then, if $\abs{z_{P\dot+\theta}-z_{Q\dot+\theta}}=\abs{z_P-z_Q}\leq 2^{k-2}\ell(P)$, it follows that $z_{Q\dot+\theta}\in (P\dot+\theta)^{(k,\theta)}$ and then, by the nestedness properties of cubes of $\mathscr D^\theta$, we have $Q\dot+\theta\subset (P\dot+\theta)^{(k,\theta)}$. In other words, $P\dot+\theta$ and $Q\dot+\theta$ have a common $k$th dyadic ancestor in $\mathscr D^\theta$. This is a desirable property that we can achieve with the help of random averaging and both parts \eqref{it:indep} and \eqref{it:2d} of Lemma \ref{lem:good} as follows:
\begin{equation*}
\begin{split}
  2^{-d} &\E\tau_{a,b;\mathscr D^\theta}^{(\gamma,k)}(f,g) \\
  &\overset{\eqref{it:2d}}=\sum_{\substack{P,Q\in\mathscr D_{[a,b)}^0 \\ 2^{k-3}<\abs{z_P-z_Q}/\ell(P)\leq 2^{k-2}}}
   \E(1_{\{P\dot+\theta\textup{ is $(k,\theta)$-good}\}})\E\tau(D_{P\dot+\theta,Q\dot+\theta}^{\gamma_1} f,D_{Q\dot+\theta,P\dot+\theta}^{\gamma_2}g) \\
  &\overset{\eqref{it:indep}}=\sum_{\substack{P,Q\in\mathscr D_{[a,b)}^0 \\ 2^{k-3}<\abs{z_P-z_Q}/\ell(P)\leq 2^{k-2}}}
   \E\Big(1_{\{P\dot+\theta\textup{ is $(k,\theta)$-good}\}} \tau(D_{P\dot+\theta,Q\dot+\theta}^{\gamma_1} f,D_{Q\dot+\theta,P\dot+\theta}^{\gamma_2}g)\Big) \\
   &=\E\sum_{\substack{P,Q\in\mathscr D_{[a,b)}^0 \\ 2^{k-3}<\abs{z_P-z_Q}/\ell(P)\leq 2^{k-2} \\ P\dot+\theta\textup{ is $(k,\theta)$-good}}}
   \tau(D_{P\dot+\theta,Q\dot+\theta}^{\gamma_1} f,D_{Q\dot+\theta,P\dot+\theta}^{\gamma_2}g)
   =:\E\tau_{a,b;\mathscr D^\theta}^{(\gamma,k;\operatorname{good})}(f,g).
\end{split}
\end{equation*}

Having reached this far, we may again ignore the presence of randomisation, and continue with the manipulation of $\tau_{a,b;\mathscr D^\theta}^{(\gamma,k;\operatorname{good})}(f,g)$ for a fixed $\mathscr D^\theta$. For the sake of brevity, we again drop the dependence on $\theta$, and consider simply
\begin{equation*}
  \tau_{a,b}^{(\gamma,k;\operatorname{good})}(f,g)
  =\sum_{\substack{P,Q\in\mathscr D_{[a,b)} \\ 2^{k-3}<\abs{z_P-z_Q}/\ell(P)\leq 2^{k-2} \\ P\textup{ is $k$-good}}}
   \tau(D_{P,Q}^{\gamma_1} f,D_{Q,P}^{\gamma_2}g)
\end{equation*}

As we already observed, for any pair $(P,Q)$ appearing in the summation above, there must be a common $k$th dyadic ancestor $S$. Thus we may reorganise the sum under these containing $S\in\mathscr D_{[a-k,b-k)}$ to arrive at
\begin{equation*}
\begin{split}
  \tau_{a,b}^{(\gamma,k;\operatorname{good})}(f,g) &=\sum_{S\in\mathscr D_{[a-k,b-k)}} \tau_{S}^{(\gamma,k)}(f,g), \\
  \tau_{S}^{(\gamma,k)}(f,g) &:=
  \sum_{\substack{P,Q\in\mathscr D; P^{(k)}=Q^{(k)}=S \\ 2^{k-3}<\abs{z_P-z_Q}/\ell(P)\leq 2^{k-2} \\ P\textup{ is $k$-good}}}
   \tau(D_{P,Q}^{\gamma_1} f,D_{Q,P}^{\gamma_2}g),
\end{split}
\end{equation*}
where
\begin{equation*}
   \tau(D_{P,Q}^{\gamma_1} f,D_{Q,P}^{\gamma_2}g)
   =\begin{cases} \sum_{\alpha,\beta\in\{0,1\}^d\setminus\{0\}}\pair{f}{h_P^\alpha}\tau(h_P^\alpha,h_Q^\beta)\pair{h_Q^\beta}{g}, & \gamma=(1,1), \\
   \sum_{\alpha\in\{0,1\}^d\setminus\{0\}}\pair{f}{h_P^\alpha}\tau(h_P^\alpha,1_Q)(\ave{g}_Q-\ave{g}_P), & \gamma=(1,0), \\
   \sum_{\beta\in\{0,1\}^d\setminus\{0\}}(\ave{f}_P-\ave{f}_Q)\tau(1_P,h_Q^\beta)\pair{h_Q^\beta}{g}, & \gamma=(0,1).
   \end{cases}
\end{equation*}
Thanks to the separation condition $2^{k-3}<\abs{z_P-z_Q}/\ell(P)\leq 2^{k-2}$, one can estimate coefficients involving $\tau$ via the kernel representation, to the result that
\begin{equation*}
   \abs{\tau(h_P^\alpha,h_Q^\beta)}
   \lesssim\omega(2^{-k})2^{-kd}
\end{equation*}
for all $\alpha,\beta\in\{0,1\}^d$ with $(\alpha,\beta)\neq(0,0)$, where $h_P^0:=\abs{P}^{-1/2}1_P$. Defining $\mathfrak a_S^{(\gamma,k)}:=\omega(2^{-k})^{-1}\tau_S^{(\gamma,k)}$, it is then straightforward to verify the asserted size and cancellation properties of $\mathfrak a_S^{(\gamma,k)}$; see \cite[Section 12.4.b]{HNVW3} for details.
\end{proof}

\section{Conclusion of the proof of Theorem \ref{thm:T1}}\label{sec:concl}

Combining Lemmas \ref{lem:BCR}, \ref{lem:mainterm} and \ref{lem:shifts}, we have
\begin{equation}\label{eq:truncDec}
\begin{split}
  \tau(f,g)
  &=\E\sum_{P\in\mathscr D_{[a,b)}}\Big(\tau(D_P f,D_P g)+\ave{f}_P\tau(1,D_P g)+\tau(D_P f,1)\ave{g}_P\Big) \\
  &\qquad+\E\sum_{\gamma\in\{0,1\}^2\setminus\{0\}}\sum_{k=2}^\infty\omega(2^{-k})\mathfrak a^{(\gamma,k)}_{a,b;\mathscr D}(f,g)
  +\E\mathscr E_{a,b}^{\mathscr D}(f,g)
\end{split}
\end{equation}
for all $f,g\in\mathscr I$.
(There is implicit understanding that $\mathscr D=\mathscr D^\theta$ and the expectation is taken over $\theta\in\Theta=(\{0,1\}^d)^{\Z}$.)
On the first line, we have truncated versions of the Haar multiplier
\begin{equation*}
  \mathfrak h_{\mathscr D}(f,g):=
  \sum_{\alpha,\beta\in\{0,1\}^d\setminus\{0\}}\pair{f}{h_P^\alpha}\tau(h_P^\alpha,h_P^\beta)\pair{h_P^\beta}{g}
  =\sum_{P\in\mathscr D}\tau(D_P f,D_P g),
\end{equation*}
where $\abs{\tau(h_P^\alpha,h_P^\beta)}\lesssim\Norm{\tau}{wbp}+c_K$ by considerations as seen above, the paraproduct
\begin{equation*}
  \pi_{b_1,\mathscr D}(f,g)=\sum_{P\in\mathscr D}\ave{f}_P\pair{D_P b_1}{g}=\sum_{P\in\mathscr D}\ave{f}_P\tau(1,D_P g)
\end{equation*}
and the adjoint paraproduct $\pi_{b_2,\mathscr D}^*(f,g):=\pi_{b_2,\mathscr D}(g,f)$. These are known to be bounded bilinear forms on $L^p(\R^d)\times L^{p'}(\R^d)$, with bounds dominated by $\Norm{\tau}{wbp}+c_K$ and $\Norm{b_i}{\BMO(\R^d)}$, respectively, independently of $\mathscr D$.

The truncated versions with $\mathscr D_{[a,b)}$ in place of $\mathscr D$ may be realised as
\begin{equation*}
\begin{split}
  \mathfrak h_{\mathscr D_{[a,b)}}(f,g)
  &=\mathfrak h_{\mathscr D}((E_b^{\mathscr D}-E_a^{\mathscr D})f,g)
  =\mathfrak h_{\mathscr D}(f,(E_b^{\mathscr D}-E_a^{\mathscr D})g), \\
  \pi_{b_1,\mathscr D_{[a,b)}}(f,g)
  &=\pi_{b_1,\mathscr D}(f,(E_b^{\mathscr D}-E_a^{\mathscr D})g), \\
  \pi_{b_2,\mathscr D_{[a,b)}}^*(f,g)
  &=\pi_{b_2,\mathscr D}^*((E_b^{\mathscr D}-E_a^{\mathscr D})f,g) \\
\end{split}
\end{equation*}
Note $\Norm{(E_b^{\mathscr D}-E_a^{\mathscr D})f}{p}\leq 2\Norm{f}{p}$
and $(E_b^{\mathscr D}-E_a^{\mathscr D})f\to f$ in $L^p(\R^d)$ as $-a,b\to\infty$. This, the uniform boundedness of the bilinear forms $\mathfrak h_{\mathscr D}$, $\pi_{b_1,\mathscr D}$ and $\pi_{b_2,\mathscr D}^*$, and dominated convergence on $\Theta$ shows that the first line on the right of \eqref{eq:truncDec} converges to
\begin{equation*}
  \E\Big(\mathfrak h_{\mathscr D}(f,g)+\pi_{b_1,\mathscr D}(f,g)+\pi_{b_2,\mathscr D}^*(f,g)\Big)
\end{equation*}
as $-a,b\to\infty$.

On the second line of \eqref{eq:truncDec}, $\mathfrak a_{a,b;\mathscr D}^{(\gamma,k)}$ is a truncated version of the new dyadic shifts first introduced in \cite{GH:2018}. The non-truncated versions are again known to be bounded on $L^p(\R^d)\times L^{p'}(\R^d)$, with the following bound (see \cite[Theorem 12.4.26]{HNVW3}):
\begin{equation*}
\begin{split}
  \Norm{\mathfrak a_{\mathscr D}^{(\gamma,k)}}{(p,p')}
  &\lesssim (1+\log_+ k)^{\Delta(\gamma,p)},\quad
  \Delta(\gamma,p)=\begin{cases} \frac12, & \gamma=(1,1), \\ \max(\frac12,\frac1p), & \gamma=(1,0), \\ \max(\frac12,\frac{1}{p'}), & \gamma=(0,1),\end{cases} \\
  &\leq (1+\log_+ k)^{1-1/p^*},\qquad p^*:=\max(p,p').
\end{split}
\end{equation*}

\begin{remark}\label{rem:type}
In fact, \cite[Theorem 12.4.26]{HNVW3} is stated more generally for functions taking values in a UMD space $X$, in terms of numbers $1\leq t\leq p\leq q\leq\infty$ such that the space has ``type'' $t$ and ``cotype'' $q$. The scalar field $X\in\{\R,\C\}$ has every type $t\in[1,2]$ and every cotype $q\in[2,\infty]$ (see \cite[Page 55]{HNVW2}), so that one can take $t=\min(2,p)$ and $q=\max(2,p)$. With these observations, the bounds above are easily seen to be special cases of those of \cite[Theorem 12.4.26]{HNVW3}.
\end{remark}

By inspection of the definitions, the truncated versions may be realised as
\begin{equation*}
\begin{split}
  \mathfrak a_{a,b;\mathscr D}^{((1,1),k)}(f,g)
  &=\mathfrak a_{\mathscr D}^{((1,1),k)}((E_b^{\mathscr D}-E_a^{\mathscr D})f,g)
  =\mathfrak a_{\mathscr D}^{((1,1),k)}(f,(E_b^{\mathscr D}-E_a^{\mathscr D})g), \\
  \mathfrak a_{a,b;\mathscr D}^{((1,0),k)}(f,g)
  &=\mathfrak a_{\mathscr D}^{((1,0),k)}((E_b^{\mathscr D}-E_a^{\mathscr D})f,g), \\
  \mathfrak a_{a,b;\mathscr D}^{((0,1),k)}(f,g)
  &=\mathfrak a_{\mathscr D}^{((0,1),k)}(f,(E_b^{\mathscr D}-E_a^{\mathscr D})g).
\end{split}
\end{equation*}
By the same argument as before, these approach the respective $\mathfrak a_{\mathscr D}^{(\gamma,k)}(f,g)$ as $-a,b\to\infty$. To guarantee the convergence of the whole sum with these terms in \eqref{eq:truncDec} via dominated convergence, we invoke the assumption that $\Norm{\omega}{\operatorname{Dini}^{1-\frac{1}{p^*}}}<\infty$, as
\begin{equation*}
  \sum_{k=2}^\infty\omega(2^{-k})\Norm{\mathfrak a_{\mathscr D}^{(\gamma,k)}}{(p,p')}
  \lesssim  \sum_{k=2}^\infty\omega(2^{-k})k^{1-\frac{1}{p^*}}
  \eqsim\Norm{\omega}{\operatorname{Dini}^{1-\frac{1}{p^*}}}.
\end{equation*}
Hence
\begin{equation*}
  \E\sum_{\gamma\in\{0,1\}^2\setminus\{0\}}\sum_{k=2}^\infty\omega(2^{-k})\mathfrak a^{(\gamma,k)}_{a,b;\mathscr D}(f,g)
  \underset{-a,b\to\infty}{\longrightarrow}\E\sum_{\gamma\in\{0,1\}^2\setminus\{0\}}\sum_{k=2}^\infty\omega(2^{-k})\mathfrak a^{(\gamma,k)}_{\mathscr D}(f,g).
\end{equation*}

Finally, we have verified in Corollary \ref{cor:BCRerrorToZero} that
\begin{equation*}
  \mathscr E_{a,b}^{\mathscr D}(f,g)\underset{-a,b\to\infty}{\longrightarrow} 0
\end{equation*}
for all $f,g\in\mathscr I$, and an inspection of the proof shows that the left hand side is dominated quantities independent of $\mathscr D$ that converge to $0$. Thus the same limit is also valid for the expectation of the left-hand side by dominated convergence.

Summarising the discussion, taking the relevant limits in \eqref{eq:truncDec} with the help of dominated convergence, we deduce that
\begin{equation}\label{eq:nontruncDec}
\begin{split}
  \tau(f,g)
  &=\E\Big(\mathfrak h_{\mathscr D}(f,g)+\pi_{b_1,\mathscr D}(f,g)+\pi_{b_2,\mathscr D}^*(f,g)\Big) \\
  &\qquad+\E\sum_{\gamma\in\{0,1\}^2\setminus\{0\}}\sum_{k=2}^\infty\omega(2^{-k})\mathfrak a^{(\gamma,k)}_{\mathscr D}(f,g)
\end{split}
\end{equation}
for all $f,g\in\mathscr I$. On the other hand, we have already observed the boundedness on $L^p(\R^d)\times L^{p'}(\R^d)$, with explicit dependence of the norms on $k$, of all the bilinear forms appearing on the right, and this immediately implies that
\begin{equation*}
  \abs{\tau(f,g)}\lesssim\Big(\Norm{\tau}{wbp}+c_K+\Norm{b_1}{\BMO}+\Norm{b_2}{\BMO}+\Norm{\omega}{\operatorname{Dini}^{1-\frac{1}{p^*}}}\Big)
     \Norm{f}{p}\Norm{g}{p'}.
\end{equation*}
Thus $\tau$ extends to a bounded bilinear form on $L^p(\R^d)\times L^{p'}(\R^d)$, and hence there is a unique $T\in\bddlin(L^p(\R^d))$ such that
\begin{equation*}
  \tau(f,g)=\pair{Tf}{g}
\end{equation*}
for all $f,g\in\mathscr I$. Of course, we similarly have operators $\mathfrak H_{\mathscr D},\Pi_{b_1,\mathscr D},\Pi_{b_2,\mathscr D}^*,A^{(\gamma,k)}_{\mathscr D}\in\bddlin(L^p(\R^d))$ corresponding to the bilinear forms $\mathfrak h_{\mathscr D},\pi_{b_1,\mathscr D},\pi_{b_2,\mathscr D}^*,\mathfrak a^{(\gamma,k)}_{\mathscr D}$.

The dyadic representation formula
\begin{equation}\label{eq:opDec}
\begin{split}
  \pair{Tf}{g}
  &=\E\Big(\pair{\mathfrak H_{\mathscr D}f}{g}+\pair{\Pi_{b_1,\mathscr D}f}{g}+\pair{\Pi_{b_2,\mathscr D}^*f}{g}\Big) \\
  &\qquad+\E\sum_{\gamma\in\{0,1\}^2\setminus\{0\}}\sum_{k=2}^\infty\omega(2^{-k})\pair{A^{(\gamma,k)}_{\mathscr D}f}{g}
\end{split}
\end{equation}
for $f,g\in\mathscr I$ is just a restatement of \eqref{eq:nontruncDec}. From this identity on the dense subspace $\mathscr I$ of both $L^p(\R^d)$ and $L^{p'}(\R^d)$, the same identity for general $f\in L^p(\R^d)$ and $g\in L^{p'}(\R^d)$ follows by continuity of both sides with respect to the relevant norms. This concludes the proof of Theorem \ref{thm:T1}.

\section{On an extension to UMD spaces}\label{sec:UMD}

While we have stated and proved Theorem \ref{thm:T1} for scalar-valued functions, we note that everything extends, {\em mutatis mutandis}, to the setting of bilinear forms $\tau:\mathscr I\times\mathscr I\to\bddlin(X,Y)$ and operators $T\in\bddlin(L^p(\R^d;X))$ such that $\pair{T(f\otimes x)}{g\otimes y^*}=\pair{\tau(f,g)x}{y^*}$ for all functions $f,g\in\mathscr I$ and vectors $x\in X$ and $y^*\in Y^*$, where $X$ and $Y$ are Banach spaces with the so-called UMD (unconditionality of martingale differences) property. (See \cite[Chapter 4]{HNVW1} for details of this notion.) In fact, the identities throughout this paper are exactly the same for scalar or vector-valued functions, whereas the estimates, which we mostly quoted from \cite[Chapter 12]{HNVW3}, are there developed in this very generality of UMD space -valued functions. Some further details worth mentioning are:
\begin{itemize}
  \item the required exponent in the Dini-condition will also depend on the ``type'' and ``cotype'' of spaces $X$ and $Y$ (see Remark \ref{rem:type}), and
  \item one needs to incorporate ``$R$-boundedness'' (see \cite[Chapter 8]{HNVW2} for details of this notion) into the weak boundedness property and the kernel estimates. ($R$-boundedness reduces to uniform boundedness for scalar-valued functions, but is a stronger property in general, which explains the absence of such condition in our scalar-valued considerations.)
\end{itemize}
Last but not least, the boundedness of paraproducts is a highly non-trivial and not completely understood question in the setting of operator-valued symbols. Thus, in the UMD space -valued dyadic representation \cite[Theorem 12.4.27]{HNVW3}, the classical $T(1)$ condition ``$\tau(1,g)=\pair{b_1}{g}$ and $\tau(f,1)=\pair{f}{b_2}$ for some $b_i\in\BMO(\R^d)$'', as in Theorem \ref{thm:T1}, is replaced by the condition that, for all $\mathscr D=\mathscr D^\theta$, where
\begin{equation*}
  \theta\in(\{0,1\}^d)^{\Z}_0:=\{\vartheta=(\vartheta_j)_{j\in\Z}\in(\{0,1\}^d)^{\Z}:\vartheta_j=0\text{ for all large enough }j\},
\end{equation*}
the bilinear form of the ``bi-paraproduct''
\begin{equation*}
  \lambda_{\mathscr D}(f,g):=\sum_{P\in\mathscr D}[\ave{f}_P\tau(1,D_P g)+\tau(D_P f,1)\ave{g}_P]
\end{equation*}
extends to a bounded linear operator $\Lambda_{\mathscr D}\in\bddlin(L^p(\R^d;X),L^p(\R^d;Y))$ such that
\begin{equation*}
  \pair{\Lambda_{\mathscr D}(f\otimes x)}{g\otimes y^*)}=\pair{\lambda_{\mathscr D}(f,g)x}{y^*}
\end{equation*}
for all $f,g\in\mathscr I$ and $x\in X$ and $y^*\in Y^*$, with operator norms of $\Lambda_{\mathscr D}$ uniformly bounded in $\mathscr D$.

Aside from the fact that this condition might be less readily checkable than the classical BMO conditions, this condition is also slightly unsatisfactory from the philosophical point of view in that we ask for a uniform bound over all $\mathscr D$, whereas the averaged form of the representation formula suggests than an average over $\mathscr D$ should suffice. The technical reason why such a formulation was not achieved in \cite[Chapter 12]{HNVW3} is the approach via an initial finitary truncation, as discussed in the Introduction: in this setting, the expectation (or the underlying probability space) changes as a function of the truncation parameter (see \eqref{eq:DRTlimit}), so that there is no meaningful expectation ``$\E\Lambda_{\mathscr D}$''. (In fact, the ``relevant'' $\mathscr D=\mathscr D^\theta$ in \cite[Theorem 12.4.27]{HNVW3} consist of those $\theta\in(\{0,1\}^d)^{\Z}$ such that $\theta_j\equiv 0$ for all large enough $j$. For each fixed $n$, there is a well-defined expectation on $(\{0,1\}^d)^{\Z_{\leq n}}$, but not on the union of these sets.)

However, this issue is avoided with the present representation of Theorem \ref{thm:T1}, where the expectation is always over $(\{0,1\}^d)^{\Z}$, and hence a bilinear form
\begin{equation*}
  \ave{\lambda}(f,g):=\E\lambda_{\mathscr D}(f,g)
\end{equation*}
on $\mathscr I\times\mathscr I$ is well-defined. This would allow a variant of \cite[Theorem 12.4.27]{HNVW3}, where the condition that every $\lambda_{\mathscr D}$ extends to an operator in $\bddlin(L^p(\R^d;X),L^p(\R^d;Y))$ by the more satisfactory version that the average $\ave{\lambda}$ alone has this property.

Moreover, recalling the definition of the shifted dyadic systems $\mathscr D$, the expectation over the probability space $(\{0,1\}^d)^{\Z}$ may be replaced by more concrete expressions:
\begin{equation*}
\begin{split}
  \ave{\lambda}(f,g)
  &=\E\sum_{P\in\mathscr D^0}[\ave{f}_{P\dot+\theta}\tau(1,D_{P\dot+\theta} g)+\tau(D_{P\dot+\theta} f,1)\ave{g}_{P\dot+\theta}] \\
  &=\sum_{P\in\mathscr D^0}\E[\ave{f}_{P\dot+\theta}\tau(1,D_{P\dot+\theta} g)+\tau(D_{P\dot+\theta} f,1)\ave{g}_{P\dot+\theta}] \\
  &=\sum_{P\in\mathscr D^0}\abs{P}^{-1}\int_{[0,\ell(P))^d}[\ave{f}_{P+u}\tau(1,D_{P+u} g)+\tau(D_{P+u} f,1)\ave{g}_{P+u}]\ud u \\
  &=\sum_{j\in\Z}\sum_{k\in\Z^d}\int_{[0,1)^d} [\ave{f}_{2^{-j}([0,1)^d+k+v)}\tau(1,D_{2^{-j}([0,1)^d+k+v)} g) \\
  &\qquad\qquad\qquad  +\tau(D_{2^{-j}([0,1)^d+k+v)} f,1)\ave{g}_{2^{-j}([0,1)^d+k+v)}]\ud v \\
  &=\sum_{j\in\Z}\int_{\R^d} [\ave{f}_{2^{-j}([0,1)^d+y)}\tau(1,D_{2^{-j}([0,1)^d+y)} g) \\
  &\qquad\qquad\qquad  +\tau(D_{2^{-j}([0,1)^d+y)} f,1)\ave{g}_{2^{-j}([0,1)^d+y)}]\ud y.
\end{split}
\end{equation*}
Computations at least superficially similar to the one above appear in \cite{Bourgain:86}, where they are used to deduce the (translation-invariant) Littlewood--Paley inequalities from (non-translation-invariant) dyadic martingale inequalities via averaging. As of the time of writing, we do not know whether these considerations may be used to obtain some new results for vector-valued paraproducts, but this seems like a topic worthy of some further exploration.


\begin{thebibliography}{10}

\bibitem{AMV:prod}
E.~Airta, H.~Martikainen, and E.~Vuorinen.
\newblock Product space singular integrals with mild kernel regularity.
\newblock {\em J. Geom. Anal.}, 32(1):Paper No. 24, 49, 2022.

\bibitem{AMV:UMD}
E.~Airta, H.~Martikainen, and E.~Vuorinen.
\newblock U{MD}-extensions of {C}alder\'{o}n-{Z}ygmund operators with mild
  kernel regularity.
\newblock {\em J. Fourier Anal. Appl.}, 28(4):Paper No. 64, 40, 2022.

\bibitem{AHMTT}
P.~Auscher, S.~Hofmann, C.~Muscalu, T.~Tao, and C.~Thiele.
\newblock Carleson measures, trees, extrapolation, and {$T(b)$} theorems.
\newblock {\em Publ. Mat.}, 46(2):257--325, 2002.

\bibitem{BCR}
G.~Beylkin, R.~Coifman, and V.~Rokhlin.
\newblock Fast wavelet transforms and numerical algorithms. {I}.
\newblock {\em Comm. Pure Appl. Math.}, 44(2):141--183, 1991.

\bibitem{Bourgain:86}
J.~Bourgain.
\newblock Vector-valued singular integrals and the {$H^1$}-{BMO} duality.
\newblock In {\em Probability theory and harmonic analysis ({C}leveland,
  {O}hio, 1983)}, volume~98 of {\em Monogr. Textbooks Pure Appl. Math.}, pages
  1--19. Dekker, New York, 1986.

\bibitem{DJ:T1}
G.~David and J.-L. Journ\'{e}.
\newblock A boundedness criterion for generalized {C}alder\'{o}n-{Z}ygmund
  operators.
\newblock {\em Ann. of Math. (2)}, 120(2):371--397, 1984.

\bibitem{DKPS:24}
K.~Domelevo, S.~Kakaroumpas, S.~Petermichl, and O.~Soler~i Gibert.
\newblock Boundedness of {J}ourn\'{e} operators with matrix weights.
\newblock {\em J. Math. Anal. Appl.}, 532(2):Paper No. 127956, 2024.

\bibitem{FP:97}
R.~Fefferman and J.~Pipher.
\newblock Multiparameter operators and sharp weighted inequalities.
\newblock {\em Amer. J. Math.}, 119(2):337--369, 1997.

\bibitem{Figiel:90}
T.~Figiel.
\newblock Singular integral operators: a martingale approach.
\newblock In {\em Geometry of {B}anach spaces ({S}trobl, 1989)}, volume 158 of
  {\em London Math. Soc. Lecture Note Ser.}, pages 95--110. Cambridge Univ.
  Press, Cambridge, 1990.

\bibitem{GH:2018}
A.~Grau de~la Herr\'{a}n and T.~Hyt\"{o}nen.
\newblock Dyadic representation and boundedness of nonhomogeneous
  {C}alder\'{o}n-{Z}ygmund operators with mild kernel regularity.
\newblock {\em Michigan Math. J.}, 67(4):757--786, 2018.

\bibitem{HLLT:19}
Y.~Han, J.~Li, C.-C. Lin, and C.~Tan.
\newblock Singular integrals associated with {Z}ygmund dilations.
\newblock {\em J. Geom. Anal.}, 29(3):2410--2455, 2019.

\bibitem{HH:2016}
T.~S. H\"{a}nninen and T.~Hyt\"{o}nen.
\newblock Operator-valued dyadic shifts and the {$T(1)$} theorem.
\newblock {\em Monatsh. Math.}, 180(2):213--253, 2016.

\bibitem{HLW:17}
I.~Holmes, M.~T. Lacey, and B.~D. Wick.
\newblock Commutators in the two-weight setting.
\newblock {\em Math. Ann.}, 367(1-2):51--80, 2017.

\bibitem{Hytonen:Pet}
T.~Hyt\"{o}nen.
\newblock On {P}etermichl's dyadic shift and the {H}ilbert transform.
\newblock {\em C. R. Math. Acad. Sci. Paris}, 346(21-22):1133--1136, 2008.

\bibitem{Hytonen:A2}
T.~Hyt\"{o}nen.
\newblock The sharp weighted bound for general {C}alder\'{o}n-{Z}ygmund
  operators.
\newblock {\em Ann. of Math. (2)}, 175(3):1473--1506, 2012.

\bibitem{Hytonen:Expo}
T.~Hyt\"{o}nen.
\newblock Representation of singular integrals by dyadic operators, and the
  {$A_2$} theorem.
\newblock {\em Expo. Math.}, 35(2):166--205, 2017.

\bibitem{HLMV:Zygmund}
T.~Hyt{\"o}nen, K.~Li, H.~Martikainen, and E.~Vuorinen.
\newblock Multiresolution analysis and {Z}ygmund dilations.
\newblock {\em Amer. J. Math.}, to appear.
\newblock Preprint, arXiv:2203.15777, 2022.

\bibitem{HPTV}
T.~Hyt\"{o}nen, C.~P\'{e}rez, S.~Treil, and A.~Volberg.
\newblock Sharp weighted estimates for dyadic shifts and the {$A_2$}
  conjecture.
\newblock {\em J. Reine Angew. Math.}, 687:43--86, 2014.

\bibitem{HNVW1}
T.~Hyt\"{o}nen, J.~van Neerven, M.~Veraar, and L.~Weis.
\newblock {\em Analysis in {B}anach spaces. {V}ol. {I}. {M}artingales and
  {L}ittlewood-{P}aley theory}, volume~63 of {\em Ergebnisse der Mathematik und
  ihrer Grenzgebiete. 3. Folge}.
\newblock Springer, Cham, 2016.

\bibitem{HNVW2}
T.~Hyt\"{o}nen, J.~van Neerven, M.~Veraar, and L.~Weis.
\newblock {\em Analysis in {B}anach spaces. {V}ol. {II}. {P}robabilistic
  methods and operator theory}, volume~67 of {\em Ergebnisse der Mathematik und
  ihrer Grenzgebiete. 3. Folge}.
\newblock Springer, Cham, 2017.

\bibitem{HNVW3}
T.~Hyt\"{o}nen, J.~van Neerven, M.~Veraar, and L.~Weis.
\newblock {\em Analysis in {B}anach spaces. {V}ol. {III}. {H}armonic analysis
  and spectral theory}, volume~76 of {\em Ergebnisse der Mathematik und ihrer
  Grenzgebiete. 3. Folge}.
\newblock Springer, Cham, 2023.

\bibitem{Journe:85}
J.-L. Journ\'{e}.
\newblock Calder\'{o}n-{Z}ygmund operators on product spaces.
\newblock {\em Rev. Mat. Iberoamericana}, 1(3):55--91, 1985.

\bibitem{Mart:12}
H.~Martikainen.
\newblock Representation of bi-parameter singular integrals by dyadic
  operators.
\newblock {\em Adv. Math.}, 229(3):1734--1761, 2012.

\bibitem{NTV:Cauchy}
F.~Nazarov, S.~Treil, and A.~Volberg.
\newblock Cauchy integral and {C}alder\'{o}n-{Z}ygmund operators on
  nonhomogeneous spaces.
\newblock {\em Internat. Math. Res. Notices}, (15):703--726, 1997.

\bibitem{NTV:Tb}
F.~Nazarov, S.~Treil, and A.~Volberg.
\newblock The {$Tb$}-theorem on non-homogeneous spaces.
\newblock {\em Acta Math.}, 190(2):151--239, 2003.

\bibitem{Pet}
S.~Petermichl.
\newblock Dyadic shifts and a logarithmic estimate for {H}ankel operators with
  matrix symbol.
\newblock {\em C. R. Acad. Sci. Paris S\'{e}r. I Math.}, 330(6):455--460, 2000.

\bibitem{PS:2014}
S.~Pott and A.~Stoica.
\newblock Linear bounds for {C}alder\'{o}n-{Z}ygmund operators with even kernel
  on {UMD} spaces.
\newblock {\em J. Funct. Anal.}, 266(5):3303--3319, 2014.

\bibitem{Volberg:15}
A.~Volberg.
\newblock The proof of the nonhomogeneous {$T1$} theorem via averaging of
  dyadic shifts.
\newblock {\em Algebra i Analiz}, 27(3):75--94, 2015.

\bibitem{Wilson:converge}
M.~Wilson.
\newblock How fast and in what sense(s) does the {C}alder\'{o}n reproducing
  formula converge?
\newblock {\em J. Fourier Anal. Appl.}, 16(5):768--785, 2010.

\bibitem{Wilson:convergeH1}
M.~Wilson.
\newblock Convergence and stability of the {C}alder\'{o}n reproducing formula
  in {$H^1$} and {$BMO$}.
\newblock {\em J. Fourier Anal. Appl.}, 17(5):801--820, 2011.

\end{thebibliography}

\end{document}